%% file: SCHK2023.tex
\pdfoutput=1
\newif\ifpersonal
\personaltrue % comment to remove personal notes
\RequirePackage[l2tabu,orthodox]{nag} %detect whether obsolete packages are used

\documentclass{amsart}

\usepackage{amsmath,amsthm, amsfonts, amssymb, mathrsfs, bbm}
\usepackage{dsfont}

\usepackage[hidelinks]{hyperref}
\usepackage[capitalize]{cleveref} %now write \cref and it knows if it is a theorem/lemma/etc and it writes it for you.
\usepackage{tikz, tikz-cd} %for drawing, diagrams
\usetikzlibrary{positioning}
\usetikzlibrary{decorations.markings}
\usepackage{verbatim}
\usepackage{microtype,fixltx2e,lmodern} % latex technical issues
\usepackage[utf8]{inputenc} % input encoding
\usepackage[T1]{fontenc} % font encoding
\usepackage{enumerate,comment,braket,xspace,csquotes} % utilities
\usepackage[all,cmtip]{xy} % because the tikzcd options [shift left], [shift right] do not work on arXiv, we switched some diagrams to xymatrix
\usepackage{multirow}
\usepackage{stmaryrd}

\usepackage{xcolor}
\hypersetup{
	colorlinks,
	linkcolor={red!50!black},
	citecolor={blue!50!black},
	urlcolor={blue!80!black}
}

\theoremstyle{plain}
\newtheorem{proposition}{Proposition}[section]
\newtheorem{lemma}[proposition]{Lemma}
\newtheorem{theorem}[proposition]{Theorem}
\newtheorem*{theorem*}{Theorem}
\newtheorem{corollary}[proposition]{Corollary}
\newtheorem*{corollary*}{Corollary}
\newtheorem{conjecture}[proposition]{Conjecture}

\theoremstyle{definition}
\newtheorem{definition}[proposition]{Definition}
\newtheorem{remark}[proposition]{Remark}

\input{mynewcommands.tex}

\newcounter{giuseppe}
\newcounter{mattia}
\newcounter{robert}
\newcounter{giulia}
\newcommand{\GA}[1]{\stepcounter{giuseppe}{\small{\color{orange}\textbf{Giuseppe \thegiuseppe}: #1}}}

\newcommand{\isocan}{\xrightarrow{\hspace{1.85pt}\sim \hspace{1.85pt}}}

\newcommand{\Addresses}{{% additional braces for segregating \footnotesize
		\bigskip
		\footnotesize

		Giuseppe Ancona, \textsc{IRMA, Strasbourg, France}\\ \nopagebreak
		\texttt{ancona@math.unistra.fr}
		
		\medskip
		Mattia Cavicchi, \textsc{Université Bourgogne Europe, CNRS, IMB UMR 5584, F-21000 Dijon, France}\\ \nopagebreak
		\texttt{mattia.cavicchi@ube.fr}
		
		\medskip
		Robert Laterveer, \textsc{IRMA, Strasbourg, France}\\ \nopagebreak
		\texttt{laterv@math.unistra.fr}
		
		\medskip
		Giulia Sacc\`a, \textsc{Columbia University, US}\\ \nopagebreak
		\texttt{giulia@math.columbia.edu}

}}

\numberwithin{equation}{section}%equations get numbered according to the section, for example 1.5 or 5.1
\title[Lefschetz standard conjecture for some Lagrangian fibrations]{Relative and absolute Lefschetz standard conjectures for some Lagrangian fibrations}

\author{Giuseppe Ancona}
\author{Mattia Cavicchi}
\author{Robert Laterveer}
\author{Giulia Sacc\`a}

\date{\today}

\begin{document}

	\maketitle
	\begin{abstract}
		We show that the hyper-K\"ahler varieties of OG10-type constructed by Laza--Sacc\`a--Voisin (LSV) verify the Lefschetz standard conjecture.
		This is an application of a more general result, stating that certain Lagrangian fibrations verify this conjecture. The main technical assumption of this general result is that the Lagrangian fibration satisfies the hypotheses of Ng\^o's support theorem. Verifying that the LSV tenfolds do satisfy those hypotheses is of independent interest. Another point of independent interest of the paper is the definition and the study of the Lefschetz standard conjecture in the relative setting, and its relation to the classical absolute case.
		\end{abstract}

			\tableofcontents
			
%\keywords{Algebraic cycles, motive, standard conjectures, HyperK\"ahler variety}

\thanks{GA and MC are supported by ANR grant ANR--18--CE40--0017. RL is supported by ANR grant ANR--20--CE40--0023.  GS acknowledges support from NSF CAREER grant DMS-2144483 and NSF FRG grant DMS-2052750.}

	\newpage
\section{Introduction}
%\GA{f} \RL{h}   \MC{fr} \GS{fr}

The standard conjectures, formulated by Grothendieck and Bombieri in the sixties, have been hugely influential in the development of the theory of motives and algebraic cycles \cite{GrothendieckBombay,Kleiman69,KleimanSeattle}. Yet, in spite of their pervasive presence\footnote{Grothendieck presented them as the most urgent task in algebraic geometry.}, not a great deal more is known about them nowadays than in the sixties.

In this paper we will work over the complex numbers. In this setting it is known that the Lefschetz standard conjecture implies all the other standard conjectures. Let us recall its formulation.
Let $L$ be an ample divisor on a smooth projective variety  $X$  of dimension $d_X$. The hard Lefschetz theorem says that the map 
\begin{align}\label{CSTL} \cup L^n :  H^{d_X-n}(X, \bQ  ) \isocan  H^{d_X+n}(X, \bQ  )  \end{align} 
is an isomorphism. The Lefschetz standard conjecture predicts that the inverse of this map is induced by an algebraic correspondence. This project is guided by the basic question: how well does this conjecture behave under fibration?

The starting point of this note is Voisin's paper \cite{Voi20}, where the Lefschetz standard conjecture in degree 2 is proved for HK manifolds which are covered by Lagrangian subvarieties. Particularly relevant for our paper is the case of Lagrangian fibrations: in this case, Voisin proves the conjecture in degree two, under the assumption that the SYZ conjecture holds. Here, we use Voisin's idea, extending it thanks to the use of Ng\^o's Support Theorem.

To explain how this theorem enters the picture, let us assume that $X$   is endowed with a surjective morphism $f : X \rightarrow B$ to a smooth projective variety $B$ (with $f$ non-constant and non-finite). We ask ourselves how much the Lefschetz standard conjecture for $B$ and for the smooth fibers may help to show the Lefschetz standard conjecture for $X$. 
 
 Let us be more precise and let us read this through the decomposition theorem. To ease notation, we will work with the constant complex $C_X=\bQ [d_X]$ on $X$ and write $H^n(X)=H(X,C_X)=H^{d_X+n}(X, \bQ  )$.
Moreover, let $\eta$ be a relatively ample divisor on $X$, $L_B$ an ample divisor on $B$ and  $\beta$ its pullback on $X$. 
The decomposition theorem gives  in particular a decomposition 
 \begin{align}\label{BBD}  H^n(X)=\bigoplus_{a+b=n} H^a(B, ^p\!\!R^{b}f_* C_X ),\end{align} 
 where  the $ ^p\!\!R^{b}f_* C_X$ denote the perverse cohomology objects of $Rf_* C_X$,
 as well as the following isomorphisms:
\begin{align}\label{BBDeta} \cup \eta^{b} :  H^a(B, ^p\!\!R^{ -b}f_* C_X ) \isocan H^a(B, ^p\!\!R^{ +b}f_* C_X  ), \end{align} 
\begin{align}\label{BBDbeta}  \cup \beta^{a} :  H^{  -a}(B, ^p\!\!R^{b}f_*C_X) \isocan H^{  +a}(B, ^p\!\!R^{b}f_* C_X  ).\end{align} 

The combination of \eqref{BBD}, \eqref{BBDeta} and \eqref{BBDbeta} may suggest that the Lefschetz standard conjecture can be reduced to showing that the inverses of \eqref{BBDeta} and \eqref{BBDbeta}  are algebraic. Additionally, at first sight one might think that an algebraic construction of the inverse of  \eqref{BBDeta} can be reduced to the Lefschetz standard conjecture for the smooth fibers of $f$ and similarly that an algebraic construction of the inverse of  \eqref{BBDbeta} can be reduced to the Lefschetz standard conjecture for $B$.
  However, a number of subtleties appear.
\begin{enumerate}
\item The perverse sheaf $^p\!R^{b}f_* C_X$ might have direct factors supported on the singular locus of $f$. In this case the information on the smooth fibers will give no control on these direct factors.
\item Inverting the action of $L_B$ on the cohomology of  $B$ is not enough for inverting  $\beta$. Indeed $f^* H(B)\subset H(X)$ coincides only with the factor $H^*(B, ^p\!\!R^{-d_X}f_* C_X )$ in the decomposition $\eqref{BBD}$. This is the most tragic point.
\item Even if one were able to construct $\Lambda_{ \eta^{b}}$ and $\Lambda_{\beta^{a}}$ acting as the inverses of  \eqref{BBDeta} resp. \eqref{BBDbeta}, it is unclear how to put all of them together and deduce an algebraic correspondence  $\Lambda_{n}: H^{ n}(X ) \isocan H^{ -n}(X  ).$ For example one could think of the formula $ \Lambda_{n} = \sum_{a+b=n}\Lambda_{ \eta^{b}}\circ \Lambda_{\beta^{a}}$ but such a sum could lead to cancellations. It is indeed  unclear how, for example, $\Lambda_{ \eta^{b}}$ acts on perverse degrees different from $b$. Already the actions of  $\eta$ and $\beta$ are not in general bigraded with respect to the decomposition \eqref{BBD}.
\end{enumerate}
  
  In the case where $X$ is a hyper-K\"ahler 
  %\RL{We need to settle on the spelling of HK: do we want to write hyper-k\"ahler, Hyper-K\"ahler, hyper-K\"ahler ? I'm fine with all of those...  } 
  variety (hence $f$ is a Lagrangian fibration by Matsushita \cite{Matsushita99}) one can partially bypass  these problems.
  \begin{enumerate}
\item Ng\^o's support theorem \cite{Ngo10} can be used to control the support of $^p\!R^{b}f_* C_X$ and avoid problem (1) in favorable cases (see \cref{relative SCLT HK}).
\item An idea of Voisin \cite{Voi20} can be used to (roughly speaking) construct a second Lagrangian fibration inverting the role of $\eta$ and $\beta$. (Hence problem (2) reduces to problem (1).)
\item A general result of Shen and Yin shows that the operators $\eta$ and $\beta$ are actually bigraded in the setting of Lagrangian fibrations \cite{SY21}.
\end{enumerate}
These ingredients allow us to prove the following general criterion. 
\begin{theorem}\label{thm intro}
Let $X$ be a hyper-K\"ahler variety admitting a Lagrangian fibration $f:X \rightarrow B$.
Assume that:
   \begin{enumerate}
\item The Picard rank of $X$ is $2$.
\item There is only one isotropic line in the Picard group which is in the boundary of the birational K\"ahler cone (the line generated by $\beta$).
\item The fibration $f$ is an Ng\^o fibration (see Definition \ref{def ngo}).
\item The fibers of $f$ are irreducible.
\end{enumerate}
Then the Lefschetz standard conjecture holds for $X$.
\end{theorem}
    The main application is the following. Some other applications will be given in the text.
  \begin{theorem}  \label{thm LSV intro} Let $M$ be a smooth cubic fourfold, and let $X$ be the associated Lagrangian fibered hyper-K\"ahler manifold constructed in \cite{LSV17,Voisin-twisted,Sacca}. Then $X$ satisfies the standard conjectures. 
  \end{theorem}
In the same spirit as  \cite{Voi20}, we give a generalization of Theorem \ref{thm intro}.  
\begin{theorem}\label{thm intro2}
Let $X$ be a hyper-K\"ahler variety admitting a Lagrangian fibration $f:X \rightarrow B$.
Assume that:
   \begin{enumerate}
\item The Picard rank of $X$ is $2$.
\item The SYZ conjecture holds for $X$.
\item The Lagrangian fibrations on any birational model of $X$  are Ng\^o fibrations.
\item The fibers of these fibrations are irreducible.
\end{enumerate}
Then the Lefschetz standard conjecture holds for $X$.
\end{theorem}

  Let us comment  on the above theorems, their assumptions, their applications and their possible generalizations.
  
  Assumption (1) is very reasonable. Indeed,  by \cite{Matsushita-def} and local Torelli any hyper-K\"ahler variety admitting a Lagrangian fibration is the specialization of a Picard rank $2$ hyper-K\"ahler admitting a Lagrangian fibration and, furthermore,  the Lefschetz standard conjecture for a variety implies the conjecture for all of its specializations.
  
  If assumption (1) holds, then there are exactly two isotropic lines with respect to the Beauville--Bogomolov form. One, namely the line generated by the lagrangian class $\beta$, is certainly on the boundary of the birational K\"ahler cone. As far as the second isotropic line is concerned, there are two cases, depending on whether it  is also on this boundary. Theorems \ref{thm intro} and \ref{thm intro2} correspond precisely to these two cases. If the other line is also on this boundary, then by the SYZ conjecture there is a second Lagrangian fibration associated with it. Notice that this conjecture has been proved for all known deformation types of hyper-K\"ahler varieties. When there is only one line there is no second fibration, since the second ray of the birational K\"ahler cone corresponds to a birational contraction. There is, however, a second fibration up to an automorphism of the motive, as was observed by Voisin \cite{Voi20}.
   %but there is, roughly speaking, a motivic one. This circle of ideas with the two lagrangian fibrations is due to Voisin \cite{Voi20}.
  
The restrictive assumptions are (3) and (4). Note that in \cite{Voi20}, this result is proved in degree $2$ without these two assumptions on the Lagrangian fibration, since singular fibers end up not mattering for the degree $2$ case.
Assumption (3) is needed to apply Ng\^o's support theorem and assumption (4) is needed to deduce through Ng\^o's support theorem that the perverse sheaves $^p\!R^{b}f_*C_X$ have full support. For the LSV tenfolds it takes some work to verify assumption (3); this part of the paper (Section \ref{S:LSV}) is of independent interest. 
  
Despite the recent growth of interest in the use of  Ng\^o's support theorem for Lagrangian fibrations, not many are known to be  Ng\^o fibrations (i.e. to satisfy the assumptions required to implement the support theorem). It is a folklore conjecture that all Lagrangian fibrations are Ng\^o fibrations (see \cite{AF} for a recent progress). Hence, we expect that in the future our methods will be applied to more general Lagrangian fibrations.

%  In the current status of the literature there are few examples of lagrangian fibrations which are known to be Ng\^o fibrations; this is the main limitation to applications for the moment. On the other hand, the interest in using Ng\^o's support theorem for lagrangian fibrations has been recently increased \GA{do we want to cite Giulia and others? (Robert) Yes, good idea !} and it is a folklore expectation that all lagrangian fibrations should be Ng\^o fibrations. Hence, it seems reasonable to expect there will soon be more known cases and thus more applications of our criterions.
  
  Our strategy is to first deal with a relative version of the Lefschetz standard conjecture and then to pass from the relative to the classical absolute case. This relative version is new and contains some subtleties with respect to the absolute one. The correct setting for this relative version is a category of relative homological motives as defined in \cite{CoHa}. This category seems to be the correct one for dealing with problems such as the construction of algebraic classes, but of course it is very far from the conjectural abelian category of relative pure motives.
%    is inspired from the one
%  Let us mention that it takes inspiration from \cite{CoHa}, in that the setting for it is a category of  relative homological motives which is 
%   but some modifications are needed. In particular, we define a category of relative homological motives which seems to be the correct one\footnote{The correct one for this kind of problems, such as constructions of algebraic classes, but it is of course very far from the conjectured one! \GS{I find "the conjectured one" not clear, maybe we should rephrase?}}. 
   The systematical study of the relative setting  seems interesting in its own right. In a sequel to this paper we will  give a stratified version of these constructions to deal with the case when the  perverse sheaves $^p\!R^{b}f_*C_X$ have direct factors strictly supported on the base. This will allow us to get rid of assumption (4), leading to even more applications.  
  
To end this introduction let us mention that, with very different methods, Charles and Markman proved the Lefschetz standard conjecture for hyper-Kähler
varieties of {\em K3$^{[n]}$-type} \cite{ChMa}, and more recently Foster proved it for those of Kummer type in some degrees \cite{Fost}. Together with \cite{Voi20} mentioned above, where, in addition to the case of Lagrangian fibrations, the degree two case of the Lefschetz standard conjecture is proved for codimension one families of projective hyper-Kähler fourfolds that are swept by Lagrangian subvarieties, these are among the rare previous results on this conjecture and not only in the realm of hyper-Kähler varieties. 
  
%  To end this introduction let us mention that, with very different methods, Charles and Markman proved the Lefschetz standard conjecture for hyper-K\"ahler varieties of 
% {\em K3$^{[n]}$-type} \cite{ChMa} . \GS{Together with the result of Voisin  \cite{Voi20} mentioned above,} these are among the rare previous results on this conjecture and not only in the realm of hyper-K\"ahler varieties.

\subsection*{Organization of the paper}	
Section \ref{S:conventions} gives the general conventions and contains in particular  the construction of the category of relative homological motives.
In Section \ref{S:main results} we state the main results and recall the basic definitions needed.
Section \ref{S:relative Lef} gives several equivalent formulations of the relative Lefschetz standard conjecture; we also highlight the subtleties which appear in this relative setting and not in the absolute one.
In Section \ref{S:good decomp} we recall the notion of good decomposition of cohomology due to de Cataldo \cite{DeCataldo}. In the presence of such a good decomposition the subtleties explained in the previous section disappear. In general, the decompositions induced by the decomposition theorem are not good, but it turns out that they are so for Lagrangian fibrations. This is a particular instance of a result of Shen and Yin \cite{SY21} and it is explained in Section \ref{section SY}. The second tool is Ng\^o support theorem. We recall it for the convenience of the reader in Section \ref{section waf} and explain some applications.
In Section \ref{section voisin} we recall the argument of Voisin that allows us to reason as if there were a second Lagrangian fibration, and use this to conclude the proof of the main result, Theorem \ref{main thm} (Theorem \ref{thm intro2} in the introduction, of which Theorem \ref{thm intro} is a special case).
The long Section \ref{S:LSV} is dedicated to LSV tenfolds; the aim is to deduce Theorem \ref{thm LSV} (Theorem \ref{thm LSV intro} in the introduction) from our main result. The first subsection recalls their geometry, the second proves the polarizability of the underlying group scheme and the third concludes the proof of Theorem \ref{thm LSV}.

\subsection*{Acknowledgments}	
We thank Salvatore Floccari, Lie Fu, Emanuele Macr\`{\i}, Claire Voisin and Qizheng Yin for their interest in our work and for numerous fruitful discussions. We also thank the referee for constructive comments that helped to improve the paper.

% This research was partly supported by the grants ANR--18--CE40--0017 and ANR--20--CE40--0023 of the Agence National de la Recherche. GS acknowledges support from NSF CAREER grant DMS-2144483 and NSF FRG grant DMS-2052750.
% \GA{Add your supports}
 
\section{Conventions}\label{S:conventions}

Throughout the paper we will work with the following notation and conventions.

\begin{enumerate}

\item \textbf{Varieties.} We work with algebraic varieties over fields of characteristic zero, and most of the time with algebraic varieties over the complex numbers.
Whenever a field of definition different from $\bC$ appears (for example the residue field of the generic point of a variety) standard techniques allow us to reduce to $\bC$ (take a model and embed the field of definition).

 \item \textbf{Rational coefficients.} All the invariants we consider (Chow rings, cohomology,$\ldots$) will be with rational coefficients.

\item \textbf{Cohomology.} For a variety $X$ of dimension $d_X$ we  denote by $C_X=\bQ[d_X] $ the constant sheaf concentrated in degree $-d_X$. We  write $H^n(X)$ for $H^n(X,C_X)$. This is the usual singular cohomology with rational coefficients with degrees shifted so that the Poincar\'e duality is centered in zero.

\item \textbf{Derived categories.} Fix a base variety $B$, and consider either the analytic topology on $B(\bC)$, or the étale topology on $B$. We will denote by $D(B)$ the derived category of sheaves over $B$ for either one of these two topologies, and by $D^b(B)$ its full subcategory of bounded objects. We will use $\mathbb{Q}$-coefficients in the analytic case, and $\overline{\mathbb{Q}}_{\ell}$-coefficients, for $\ell$ a prime, in the étale case. The context will make clear which category we are using. We will consider either the \emph{classical} or the \emph{perverse} $t$-structure on $D(B)$ (see \cite{BBD} for generalities). For an integer $n \in \bZ$, the classical cohomology objects of an object $K$ in $D(B)$ will be denoted by $H^nK$. If $f: Z \rightarrow B$ is a morphism and $K$ a complex over $Z$, write $F$ for either one of the functors $f_*, f_!$. We will denote by $R^n F(K)$ the classical cohomology objects of the complex $RF(K)$, and by $^p\! R^n F(K)$ its perverse cohomology objects.
 
 \item \textbf{HK.} Our center of interest will be (projective) hyper-K\"ahler varieties. Such a variety will be called a HK variety or simply a HK. 

\item \textbf{Motives.} We will work with the category $\Mot$ of homological motives defined over $\bC$ and  with rational coefficients.
The category comes equipped with a realization functor
\[R : \Mot  \longrightarrow \Vect. \]
 to the category of rational vector spaces.
(The latter are endowed with a Hodge structure, which we will not need.)
For generalities see \cite[\S 4]{Andmot}.

\item \textbf{Tate twists.} \label{Tate twists} 
We will make the brutal choice  of fixing an isomorphism $\bQ \cong \bQ(1)$ in cohomology, hence ignoring the Tate twists.
This is done so that the operation of cup-producing with a given algebraic class is an endomorphism of the cohomology ring. 
See the discussion in \cite[5.2.3]{Andmot}.

However, we will keep track of the Tate twists for motives. Hence the realization of the motive $\one(1)$ is the complex $\bQ (1)[2]$, which by our convention is identified with the constant complex $\bQ$ in degree $-2$.

\item \textbf{Relative motives.} 
Following Corti--Hanamura \cite{CoHa} (cf. also \cite[Chapter 8]{MNP}), there is a category $\CHM(B)$ of relative Chow motives over a base $B$. Whenever $B$ is projective, there is a commutative square of functors
 \[ \begin{tikzcd}
\CHM(B) \arrow[d] \arrow[r,"R_B"]& D^b(B) \ar[d,"H^*"]\\
\CHM(\bC) \ar[r,"R"] & \Vect.
\end{tikzcd}
 \]
  Note that the horizontal functors are not   faithful. To fix this, we work up to homological equivalence. This means we can define $\Mot(B)$ as (the pseudo-abelian envelop of) the quotient of the motivic categories with respect to the tensor ideal given by the kernel of the realization functors. 
By construction, this gives  the commutative diagram
 \[ \begin{tikzcd}
\Mot(B) \arrow[d] \arrow[r,"R_B"]& D^b(B) \ar[d,"H^*"]\\
\Mot \ar[r,"R"] & \Vect.
\end{tikzcd}
 \]
 where the horizontal functors are now faithful by construction.
 
 %Notice that Corti and Hanamura give a different definition of homological motive that is a quotient of ours \cite{CoHa}.
% With their definition the above diagram would not commute.

 	\end{enumerate}
	
\section{Statement of the main result}\label{S:main results}

In this section we state our main result, Theorem \ref{main}, and recall the main definitions which are needed. The result will be restated and proven as Theorem \ref{main thm} in Section \ref{section voisin}, once the necessary ingredients will have been developed in Sections \ref{S:relative Lef} to \ref{section voisin}.

%\begin{theorem}\label{main} Let $X$ be a HK variety admitting a Lagrangian fibration. Assume moreover the following:
%
%\begin{enumerate}
%\item $X$ has Picard number 2;
%\item the SYZ conjecture (Conjecture \ref{conj SYZ}) holds for $X$; 
%\item for any Lagrangian class in $H^2(X)$ (cf. Definition \ref{def lagrangian}), the associated Lagrangian fibration is an Ng\^o fibration (cf. Definition \ref{def ngo}) and has irreducible fibers.
%\end{enumerate}
%
%Then the standard conjectures hold for $X$.
%\end{theorem}

\begin{theorem}\label{main}
Let $X$ be a hyper-K\"ahler variety admitting a Lagrangian fibration $f:X \rightarrow B$.
Assume that:
   \begin{enumerate}
\item The Picard rank of $X$ is $2$.
%\item there is only one class in $H^2(X)$ which is lagrangian (Definition \ref{def lagrangian}), or 
\item The SYZ conjecture (Conjecture \ref{conj SYZ}) holds for $X$. \item The Lagrangian fibrations on any birational model of $X$  are Ng\^o fibrations (cf. Definition \ref{def ngo}).
\item The fibers of these fibrations are irreducible.
\end{enumerate}
Then the Lefschetz standard conjecture holds for $X$ (and hence all the standard conjectures hold for $X$).
\end{theorem}

This is Theorem \ref{thm intro2} in the introduction; Theorem \ref{thm intro} then follows as a special case. Let us recall some of the concepts appearing in Theorem \ref{main}.

\begin{definition}\label{definition HK} A smooth projective variety $Y$ is called hyper-K\"ahler (HK) if it is simply-connected and 
  \[ H^0(Y,\Omega^2_Y)=\bC[\sigma_Y]\ ,\] 
  where $\sigma_Y$ is a nowhere-degenerate two-form.
  
\end{definition}

\begin{remark}  In every even complex dimension greater or equal to $ 4$, there are two series of known examples: those of K3$^{[n]}$-type, which are deformation equivalent to the Hilbert schemes of $n$ points on a K3 surface, and those of generalized Kummer type, which are deformation equivalent to generalized Kummer varieties. In addition to these, there are two other deformation classes in dimensions $6$ and $10$. These deformation classes are referred to as OG$6$, respectively OG$10$, and hyper-K\"ahler manifolds in these deformation classes are called of {\em OG6-type}, respectively  of  {\em OG10-type}.
%
%There are only four deformation classes of HK varieties known up till now: Hilbert schemes of points on K3 surfaces and their deformations (the so-called {\em K3$^{[n]}$-type}), generalized Kummer varieties and their deformations (the so-called {\em generalized Kummer type}), certain crepant resolutions of singular symplectic varieties of dimension 6 and 10 constructed by O'Grady and their deformations (these are called {\em OG6-type} resp. {\em OG10-type}).
\end{remark}

\begin{definition}\label{definition lagrangian} Let $Y$ be a HK variety. A {\em Lagrangian fibration\/} is a surjective morphism with connected fibers $f\colon Y\to B$ to a normal variety such that $0<\dim B < \dim X$.
\end{definition}

\begin{remark} 
By \cite{Matsushita99}, the smooth fibers of the fibration are Lagrangian (i.e. they are maximal isotropic submanifolds of $Y$), in particular,  $\dim B={1\over 2}\dim X$. In fact, more is true, namely that every fiber is Lagrangian in the sense that the smooth locus of the reduced underlying variety is a Lagrangian submanifold \cite{Matsushita2000}. In particular, $f$ is equidimensional. A long-standing expectation is that the base of a Lagrangian fibration is a projective space; this is known to be true when $\dim Y=4$ \cite{HX} and also known to hold under the condition that $B$ is smooth \cite{Hwang}.
%It is known that the smooth loci of the fibers of a lagrangian fibration are lagrangian (i.e. the two-form $\sigma_Y$ restricts to zero); in particular, $\dim B={1\over 2}\dim X$ \GA{I don't think we can conclude this. You need to know that they are maximal lagrangian, right? (Robert) Yes indeed; for me ``lagrangian'' means ``maximal isotropic''.} and the smooth fibers are abelian varieties. A long-standing expectation is that the base of a lagrangian fibration is a projective space; this is known to be true when $\dim Y=4$ \cite{HX} and also known to hold under the condition that $B$ is smooth \cite{Hwang}.
%For more on lagrangian fibrations, cf. \cite{Matsushita99}, \cite{HM}.
\end{remark}

\begin{remark}\label{rem huy}Recall that birational HK varieties are non separated points of the moduli spaces, in the sense that they appear as central fibers of one parameter deformations that are isomorphic away from the central fiber; the choice of such families induces an isomorphism of their cohomology rings and motives \cite{Huy,Riess}. As a consequence, the Lefschetz standard conjecture for one HK manifold implies the Lefschetz standard conjecture for all of its birational HK models. Recall also that birational HK manifolds are isomorphic in codimension $2$ and thus have canonically isomorphic second cohomology groups.
\end{remark}

% have isomorphic cohomology rings and motives \cite{Huy,Riess}
% are deformation equivalent
% have canonically isomorphic cohomology rings and motives \cite{Huy,Riess}.
The previous remark is implicitly used in the following definition. 
%\GA{As Giulia as pointed out, this iso is not canonical. We should say this differently, maybe see how voisin does it?}
\begin{definition}\label{def lagrangian}
Let $Y$ be a HK and $q$ be the Beauville--Bogomolov quadratic form on $H^2(Y)$. 
A divisor class $l$ on $Y$ is said to be Lagrangian if there exists a birational HK $Y'$ on which $l$ is semi-ample. This means that $Y'$ admits a Lagrangian fibration $Y' \to B'$ and that, up to the identification $H^2(Y) \cong H^2(Y')$,  $l$ is the pullback of an ample class on $B'$.

The class $l$ is said to be potentially Lagrangian if $q(l)=0$ and $l$ is in the boundary of the birational K\"ahler cone.
\end{definition}
\begin{conjecture}\label{conj SYZ} (Conjecture SYZ)
All potentially Lagrangian classes are actually Lagrangian.
\end{conjecture}

This conjecture holds true for all known examples of HK varieties:

\begin{proposition}\label{SYZ holds} Let $Y$ be a HK variety which is either of K3$^{[n]}$-type, or of generalized Kummer type, or of OG6-type, or of OG10-type. Then Conjecture \ref{conj SYZ} is true for $Y$.
\end{proposition}

\begin{proof} For HK varieties of K3$^{[n]}$-type or of generalized Kummer type, the statement is proven in \cite[Corollary 1.1]{Matsushita}.
For HK varieties of OG6 resp. OG10-type, the statement is proven in \cite[Theorem 7.2]{MongardiRap} resp. \cite[Theorem 2.2]{MongardiOno}.
\end{proof}

\section{The relative Lefschetz standard conjecture}\label{S:relative Lef}
 
Throughout this section we fix a  surjective projective morphism $f:X \longrightarrow B$ of algebraic varieties with $X$ smooth and connected. 
We will denote by $ C_X$ the constant perverse sheaf on $X$.
We also fix a relatively ample divisor $\eta$.

The goal of the section is to formulate an analogue of the Lefschetz standard conjecture in this relative setting, and to give some basic properties.

Recall that in this geometric context an algebraic cycle in $X\times_BX$ acts on the relative cohomology \cite{CoHa}, \cite[Chapter 8]{MNP}. Seen in this way it is called a relative correspondence. 

\begin{definition}\label{def SCLT}
We say that the pair $(f,\eta)$ verifies the relative Lefschetz standard conjecture (or simply that $X$ does) if for all cohomological degrees $n\in \bN$ there exists a (non-unique) relative correspondence $\gamma_n(f)$ inducing the inverse of the relative hard Lefschetz isomorphism
\[\cup \eta^n \colon\ \  ^p\!R^{-n}f_*C_X \ \cong\  ^p\!R^{+n}f_\ast C_X [2n] \]
in degree $-n$ (without any prescribed property in the other degrees).
\end{definition}

\begin{proposition}\label{criterion SCLT}
Assume that:
\begin{enumerate}
\item The fiber over the generic point of $B$ verifies the Lefschetz standard conjecture,
\item For all $n$ the perverse sheaf $ ^p\! R^{n}f_\ast C_X$ is the intersection complex of a local system. 
\end{enumerate}
Then the pair $(f,\eta)$ verifies the relative Lefschetz standard conjecture (\cref{def SCLT}).
\end{proposition}
\begin{proof}
By the first hypothesis, there exist correspondences $\gamma_n$  on the square of the generic fiber acting as the inverse of the Lefschetz operators. 
For each $n$ we define $\gamma_n(f)$ to be any relative correspondence extending   $\gamma_n$ (for example the closure of $\gamma_n$ in $X \times_B X$). We claim that such a $\gamma_n(f)$ satisfies the properties of \cref{def SCLT}.

We need to check the equalities 
  \[  \begin{split} \bigl(\gamma_n(f) \circ \cup \eta^n\bigr)\vert_{ ^p\! R^{-n}f_*C_X } &= \id\vert_{  ^p\! R^{-n}f_*C_X }\ ,\\
       \bigl( \cup \eta^n \circ \gamma_n(f)  \bigr)\vert_{  ^p\! R^{n}f_*C_X } &= \id\vert_{  ^p\! R^{n}f_*C_X } \ .\\
       \end{split}\]
Because of the second hypothesis, both equalities live in an algebra of the form $\End(IC(\mathcal{L}))$, with $\mathcal{L}$ a semisimple local system on an open $U$ of the base $B$.
Since intersection complexes of a simple local system are simple \cite[Théorème 4.3.1]{BBD}, it follows from Schur's lemma that restriction from $B$ to $U$ induces an injective map 
\begin{equation}\label{junlian shen}
\End(IC(\mathcal{L}))\ \hookrightarrow\ \End(\mathcal{L})\ ,\end{equation}
and so the equality can be checked after restriction to $U$.
Finally, since $\mathcal{L}$ is a local system, the equalities can be checked after restriction to any fiber and in particular to the generic fiber.
\end{proof}

\begin{proposition}\label{equiv SCLT}
The  pair $(f,\eta)$ verifies the relative Lefschetz standard conjecture if and only if the motive
$\mathfrak{h}(X/B)\in \Mot(B)$ admits a decomposition 
\begin{align}\label{relative kunneth}
\mathfrak{h}(X/B) = \bigoplus_n \mathfrak{h}^n(X/B)\ ,
\end{align}
satisfying
\begin{enumerate}
\item The realization of $\mathfrak{h}^n(X/B)$ is   concentrated in perverse  degree $n$,
\item There is an isomorphism $\mathfrak{h}^n(X/B)\cong \mathfrak{h}^{-n}(X/B)(-n)$ for all $n$.
\end{enumerate}  
In particular, the fact that the pair $(f,\eta)$ verifies the relative Lefschetz standard conjecture does not depend on the fixed relatively ample divisor $\eta$.
\end{proposition}
\begin{proof}
Let us suppose that $(f,\eta)$ verifies the relative Lefschetz standard conjecture. We show by descending induction on $N\geq 0$ that there is a motivic decomposition
\begin{align}\label{decomp induction}
\mathfrak{h}(X/B) = \mathfrak{h}^{-N\geq n\geq N} (X/B) \oplus \bigoplus_{n<-N \,\textrm{or}\, n>N} \mathfrak{h}^n(X/B)\ ,
\end{align}
such that properties (1) and (2) of the statement are satisfied and $R_B(\mathfrak{h}^{-N\geq n\geq N} (X/B))$ is concentrated in degrees between $-N$ and $N$. To show the inductive step, let us decompose further in degree $-N$ and $+N$. Consider the algebraic correspondences $\gamma_N(f)$ and $\cup \eta^N $ from Definition \ref{def SCLT} and consider the projector $p$ defining $ \mathfrak{h}^{-N\geq n\geq N} (X/B) $ inside $\mathfrak{h}(X/B)$ through the decomposition \eqref{decomp induction}. Consider the endomorphism $q=p \circ \gamma_N(f) \circ  p  \circ  (\cup \eta^N)  \circ  p$ of $ \mathfrak{h}^{-N\geq n\geq N} (X/B) $.   Let us look at its realization $R_B(q)$, which  we write as a matrix with respect to the perverse degrees (for a fixed non canonical decomposition in the derived category). By construction this is an endomorphism of a complex concentrated in degrees between $-N$ and $+N$.  The axioms of a t-structure imply that $R_B(q)$ cannot send a given degree to a higher one so the matrix associated with $R_B(q)$ is upper triangular. Let us look at the diagonal. In degree $-N$  it is  the identity by Definition \ref{def SCLT} and in higher degrees the diagonal is zero  as  already  $p  \circ  (\cup \eta^N)  \circ  p$  is so. In conclusion $R_B(q)$ is an upper triangular matrix whose diagonal is the identity in degree $-N$ and zero otherwise. Hence the powers $R_B(q)^r$ stabilize for $r$ big enough ($r\geq 2N$), becoming projectors. We deduce that $\pi_{-N}=q^{2N}$ is a projector whose realization is concentrated in degree $-N$ and coincides with $ ^p\! R^{-N}f_\ast C_X$.

Arguing similarly we deduce that $\pi_{+N }=(p \circ (\cup \eta^N) \circ  p  \circ      \gamma_N(f)   \circ  p)^{2N}$ is a projector whose realization is concentrated in degree $-N$ and coincides with $ ^p\! R^{+N}f_\ast C_X$. Now a subtlety is that the two projectors $\pi_{-N}$ and $\pi_{+N}$ are not orthogonal a priori. The composition  $\pi_{+N} \circ \pi_{-N}$ is a morphism from lower degrees to upper degrees hence vanishes. However, the other composition can be nonzero. By direct computation one shows that $\pi_{-N} $ and $  \pi_{+N} - \pi_{-N}\circ \pi_{+N} $ are orthogonal projectors and that $  \pi_{+N} - \pi_{-N}\circ \pi_{+N} $  has again $ ^p\! R^{+N}f_\ast C_X$ as realization. Let $\mathfrak{h}^{+N}(X/B)$ and $\mathfrak{h}^{-N}(X/B)$ be the motives defined by these orthogonal projectors. The isomorphism $\mathfrak{h}^N(X/B)\cong \mathfrak{h}^{-N}(X/B)(-N)$  is induced by  $\cup \eta^N$ and $\gamma_N(f)$. This concludes one implication.

The other implication is easier. The map  $\cup \eta^N$ induces through the decomposition  $ \mathfrak{h}(X/B) = \bigoplus_n \mathfrak{h}^n(X/B)$ a map from 
 $\mathfrak{h}^{-N}(X/B)(-N)$ to
  $\mathfrak{h}^N(X/B)$. The goal is to invert such a map. As these two motives are isomorphic we can see this map as an endomorphism 
 $F$ of $\mathfrak{h}^N(X/B)$. The realization  $R_B(F)$ is an isomorphism. We claim that   the inverse of  $R_B(F)$  can be written as a polynomial in $R_B(F)$, hence there is an algebraic inverse to $F$   in the algebra $\bQ[F]$.
 
 To prove the claim we use a Cayley--Hamilton argument. More precisely, as any endomorphism of perverse sheaves will respect the decomposition by supports of intersection complexes, and on each stratum the generic point determines the endomorphism (see \eqref{junlian shen}), we are reduced to consider endomorphisms of vector spaces, to which classical Cayley--Hamilton applies. (Alternatively, the recourse to Cayley--Hamilton can be justified in a more highbrow way by observing that perverse sheaves form a tannakian category, hence a fiber functor reduces the problem to vector spaces.)
\end{proof}
\begin{corollary}
Let $X$ be a smooth projective variety endowed with a morphism $f:X\to B$ verifying the relative Lefschetz standard conjecture. Then its absolute motive $\mathfrak{h}(X)\in \Mot$ admits a decomposition 
\begin{align}\label{first kunneth}
\mathfrak{h}(X) = \bigoplus_n \mathfrak{h}(B,\mathfrak{h}^n(X/B))\ ,
\end{align}
such that the realization of $\mathfrak{h}(B,\mathfrak{h}^n(X/B))$ is  $H(B,^p\!R^{n}f_*C_X )$ and the motives verify $\mathfrak{h}(B,\mathfrak{h}^n(X/B))\cong \mathfrak{h}(B,\mathfrak{h}^{-n}(X/B))(-n).$

\end{corollary}
\begin{proof}
 The statement follows directly from  Proposition \ref{equiv SCLT} by applying  the functor $f_* : \Mot(B) \rightarrow \Mot $.
\end{proof}

\begin{remark}\label{rem psy}
The decompositions \eqref{relative kunneth} and  \eqref{first kunneth} are not unique in general, since their cohomological counterparts  
\begin{align}\label{decomp coho rel} Rf_*C_X = \bigoplus_n   {}^p\! R^{n}f_*C_X \end{align}
\begin{align}\label{decomp coho abs} H(X) = \bigoplus_n H(B,{}^p\! R^{n}f_*C_X ) \end{align}
 are not. Let us stress that the above statements do not prove that  \emph{any} decomposition as in \eqref{decomp coho rel} (or as in  \eqref{decomp coho abs}) is of motivic origin. Proposition \ref{equiv SCLT} only shows  that there exists  one decomposition as in \eqref{decomp coho rel}  of motivic origin.

Deligne constructed three different decompositions of the form \eqref{decomp coho rel} which are better than the others, in the sense that each of them can be characterized as the unique decomposition satisfying some additional property (for example some behavior with respect to duality) \cite{Del_dec}. We do not know if any of these three is of motivic origin in general.

Based on Deligne's work, de Cataldo \cite{DeCataldo} showed that there are five decompositions as in \eqref{decomp coho abs} with remarkable properties. Again we do not know if those are motivic in general. The goal of the next section is to show that, under some special hypotheses which are always verified by Lagrangian fibrations, it is possible to show that de Cataldo's decompositions are motivic. This will be important in the end of Section \ref{section voisin} (proof of {Theorem \ref{main thm}}), where we will use the unicity of such a decomposition.
\end{remark}

\section{Good decomposition: from relative to absolute} \label{S:good decomp}
 
In this section we  recall the notion  of good decomposition (\cref{def de cataldo}) which is due to de Cataldo \cite{DeCataldo}. We show that the relative Lefschetz standard conjecture implies that a good decomposition is motivic (cf. \cref{rem psy}). Good decompositions do not exist in general. However, as we will see in \cref{section SY}, they do exist for Lagrangian fibrations.
%The notion of good decomposition is restrictive in general, but as we will see in the \cref{section SY} this notion  applies to lagrangian fibrations.

 We keep the same notation as the previous section:
we fix a projective surjective morphism $f:X \longrightarrow B$ of algebraic varieties, with $X$ smooth and connected, and  a relatively ample divisor $\eta$.
\begin{definition}\label{def de cataldo}(De Cataldo \cite{DeCataldo})
The pair $(f,\eta)$ admits a {\em good decomposition\/} if there exists a decomposition 
\begin{align}\label{def good}
H^*(X) = \bigoplus_n H^*(B, ^p\! \! R^nf_\ast C_X)
\end{align}
%\RL{Nitpicking comment, but: partout ou il y a le petit ${}^p$ de pervers j'ai rajoute un peu d'espace devant, car si on ecrit
%par exemple $ H^*(X) = \bigoplus_n \mathbb{H}^*(B, ^pR^nf_\ast C_X)$ le $^p$ a l'air d'appartenir a la virgule. Ma solution n'est peut-etre pas ideale non plus, on pourrait aussi faire
%$ H^*(X) = \bigoplus_n \mathbb{H}^*(B, ^p\! \! R^nf_\ast C_X)$ si vous preferez, faut juste que ce soit consistent partout...}
%\GA{I prefer the second. Mattia might have a standard one ?} \MC{I also prefer the second. My standard option was $ H^*(X) = \bigoplus_n \mathbb{H}^*(B, {}^pR^nf_\ast C_X)$ but Robert's second one looks better. I have changed the notation everywhere accordingly.}
splitting the perverse filtration and such that the map induced by the cup-product with $\eta$ is homogeneous of degree two (i.e., it is 2-graded).
In this case the induced relative primitive decomposition of $H^*(X)$ is called a {\em good primitive decomposition}.
\end{definition}
\begin{remark}\label{duality good} (cf.  \cite[Proposition 2.6.3]{DeCataldo})
In general a good decomposition does not exist. Also, it may exist for some relatively ample divisors but not for all and in general it does depend on the given divisor.
If a good decomposition exists then it is unique, and it coincides with all the splittings constructed by Deligne and de Cataldo.
In particular, it is self-dual (i.e.  Poincar\'e  duality preserves the direct sum decomposition).
The  unicity and autoduality hold also for  the induced good primitive decomposition. 
\end{remark}

The notion of good decomposition is crucial in that it allows to pass from relative Lefschetz to absolute Lefschetz. This is attested by the following proposition:

\begin{proposition}\label{motivic good decomposition} Let $X$ be a smooth projective variety with a morphism $f\colon X\to B$, and let $\eta$ be a relatively ample divisor.
Assume that:
\begin{enumerate}
\item The pair $(f,\eta)$ verifies the relative Lefschetz standard conjecture (\cref{def SCLT}),
\item The pair $(f,\eta)$ admits a good decomposition (\cref{def de cataldo}).
\end{enumerate}
Then the good decomposition is motivic and each motive in the decomposition is isomorphic to its dual (up to a Tate twist).
Moreover, the same statement holds true also for the good primitive decomposition.
\end{proposition}
\begin{proof}
The proof is similar (and actually easier) to the one of \cref{equiv SCLT} and its corollary.
In this proof, we will use the  action of a relative correspondence   on  absolute cohomology induced by the functor $\Mot(B)\to\Mot$.

We construct the motives in the good decomposition by increasing induction on the degree. By the autoduality of a  good decomposition (see \cref{duality good}),
once the motive $ \mathfrak{h}(B,\mathfrak{h}^{-n}(X/B))$ is constructed  the factor $\mathfrak{h}(B,\mathfrak{h}^{+n}(X/B))$  of the decomposition is determined by taking the dual (after  an appropriate Tate twist). 
%The fact that this other motive is also part of the decomposition comes from the autoduality of a good decomposition (see \cref{duality good}).
The isomorphism between these two motives is then deduced by the first assumption.

We can therefore suppose that we have constructed the motive $\mathfrak{h}(B,\mathfrak{h}^i(X/B))$ for every cohomological degree $i$ strictly smaller than $-n$ and strictly larger than $+n$. Let $p_i$ denote  the corresponding projectors.
The goal is now to construct such a projector in degree $-n$. 
To do so, consider $\gamma_n(f)$ as in \cref{def SCLT} and define 
  \[ p:=\id_{\mathfrak{h}(X)}-\sum p_i \] 
  to be the projector that acts on $ \mathfrak{h}(X)$ by killing all  motives   already constructed (i.e., those of cohomological degree strictly smaller than $-n$ and strictly larger than $+n$).
Then, by the very definition of good decomposition, the correspondence  $p \circ \cup\eta^n\circ p$ induces an isomorphism between degree $-n$ and degree $+n$ and zero elsewhere, hence 
  \[ p_n:=p \circ \gamma_n(f) \circ p \circ \cup\eta^n\circ p\] 
  is the desired projector.

This concludes the proof for the good decomposition, and it is standard to deduce from it the case of the good primitive decomposition.
\end{proof}

\section{Ng\^o fibrations and relative Lefschetz}
\label{section waf}

We keep the notation of the previous section. In order to exhibit morphisms $f: X \rightarrow B$ verifying the assumptions of Proposition \ref{criterion SCLT}, we now introduce the concept of {\em Ng\^o fibrations.\/} An Ng\^o fibration is a special case of a \emph{weak abelian fibration}, as defined by Ng\^o \cite{Ngo10}, \cite{Ngo2}. Before giving the definitions, let us recall some terminology concerning commutative group schemes. 
\begin{definition}
Let $g:P \rightarrow B$ be a smooth commutative group $B$-scheme, with identity component $g^o:P^o \rightarrow B$. For any point $b \in B$, there exists a canonical exact sequence of connected commutative group schemes over $k(b)$, called the \emph{Chevalley dévissage}:
\[
0 \rightarrow L_b \rightarrow P_b^o \rightarrow A_b \rightarrow 0\ ,
\]
where $L_b$ is affine and connected, and maximal with these properties, and where $A_b$ is an abelian variety. 

The $\delta$-\emph{loci} are then the locally closed subvarieties $B_i$ of $B$ defined by
\[
B_i:= \{ b \in B \vert \dim_{k(b)}(L_b)=i \}
\]

Fix a prime $\ell$ and write $\overline{\mathbb{Q}}_{\ell,S}$ for the constant étale $\overline{\mathbb{Q}}_{\ell}$-adic sheaf over a base $S$ and $\overline{\mathbb{Q}}_{\ell,S}(n)$ for its $n$-th Tate twist. Let $d:= \dim_B P$. The \emph{Tate module} of $P/B$ is the constructible $\ell$-adic sheaf $T_{\ell}(P):=R^{2d-1} g_!\overline{\mathbb{Q}}_{\ell,P}(d)$ over $B$. 

We say that $P$ is \emph{polarizable} if there exists a pairing 
\[
T_{\ell}(P) \times T_{\ell}(P)\  \rightarrow\ \overline{\mathbb{Q}}_{\ell,B}(1)
\]
such that, for any $b \in B$, the null-space of the pairing at $b$ is precisely $T_{\ell}(L_b)$. 
\end{definition}
\begin{definition} \label{WAF}
A weak abelian fibration is a pair of morphisms \[(f:X \rightarrow B, g:P \rightarrow B)\] such that:
\begin{enumerate}
\item $f$ is proper;
\item $P$ is a smooth commutative group $B$-scheme;
\item there is an action $a:P \times_B X \rightarrow X$ of $P$ on $X$ over $B$;
\item $f$ and $g$ have the same pure relative dimension, denoted by $d$;
\item the action has affine stabilizers at every point $x \in X$; 
\item the Tate module $T_{\ell}(P)$ of $P/B$ is polarizable. 
\end{enumerate}
When the morphisms defining the structure of a weakly abelian fibration on $X,P,B$ as above are clear from the context, we will write $(X,B,P)$ to denote the weakly abelian fibration. 
%If, in addition, we also have the following property of $P/B$: 
%\[
%\codim B_i \geq i, \ \forall i \in \mathbb{Z}_{\geq 0}
%\]
%then we say that $(X,B,P)$ is a $\delta$-regular weakly abelian fibration. 
\end{definition}

We propose the following definition:

\begin{definition}\label{def ngo} An {\em Ng\^o fibration} is a weak abelian fibration $(X,B,P)$ with the following additional property:
  \[
\codim B_i \geq i\  \ \forall i \ \in \mathbb{Z}_{\geq 0}
\]
(NB: in the literature, an Ng\^o fibration is called a ``$\delta$-regular weak abelian fibration''). 
\end{definition}

%%
%%\begin{remark}
%%In the recent \cite[Proposition \S 40]{Schnell}, it is proved that if $X$ is a lagrangian fibration, then the above
%%
%%\end{remark}

\begin{proposition}\label{relative SCLT HK}
Let $f:X \longrightarrow B$ be an Ng\^o fibration with irreducible fibers.
Then it verifies the relative Lefschetz standard conjecture (\cref{def SCLT}).
\end{proposition}
\begin{proof}
We want to apply \cref{criterion SCLT} to this setting. 
The first assumption is satisfied because the generic fiber is (possibly after a finite extension) an abelian variety and the Lefschetz standard conjecture is known for abelian varieties \cite{Lieb} (and  is stable under finite extensions).
The second assumption is satisfied since by Ng\^o support theorem \cite{Ngo10} an Ng\^o fibration with irreducible fibers has only full supports in the decomposition theorem. \end{proof}

\begin{corollary}  Let $f\colon X\to B$ be a Beauville--Mukai integrable system, i.e. the relative compactified  Jacobian of a linear system of curves on a K3 surface \cite{Mukai84}, \cite{Beauville91}. Assume that $f$ has irreducible fibers, i.e. that all the curves in the linear system are integral.
Then $f\colon X\to B$ verifies the relative Lefschetz standard conjecture.
\end{corollary}

\begin{proof} Since $f$ is a Lagrangian fibration, it satisfies all the requirements in the definition of Ng\^o fibration except possibly the polarizability (cf. the proof of Proposition \ref{LSV_Ngo} below). The polarizability is proven by Ng\^o for compactified relative Jacobians of curves \cite[page 76]{Ngo10} (cf. also \cite[proof of Theorem 3.3.1]{DeC17}). It follows that $f$ is an Ng\^o fibration and Proposition \ref{relative SCLT HK} applies.
\end{proof}
In addition to the above corollary, \cref{relative SCLT HK} applies also to LSV tenfolds, as we will explain in Section \ref{S:LSV}.

\section{The Shen--Yin decomposition is good}\label{section SY}  
 The object of the present section is to recall a remarkable result of Shen and Yin concerning Lagrangian fibrations, which implies in particular the existence of a good decomposition (see \cref{def de cataldo}).

Throughout this section, we fix a Lagrangian fibration $f:X \longrightarrow B$. The dimension of $X$ will be denoted by $\dim X = 2r$.
Let $q$ be the Beauville--Bogomolov form, let $\beta$ be the Lagrangian class (i.e. the pull-back of an ample divisor on $B$) and let $\eta$  be a relatively ample divisor.
Note that $q(\beta)=0$, and that we can (and will) modify $\eta$ in a unique way by adding a $\mathbb Q$-multiple of $\beta$ in order to have the vanishing $q(\eta)=0$ \cite[(22)]{SY21}. Up to clearing denominators, we can further assume that $\eta$ is an integral class.   Then, as first observed by Verbitsky, the vanishing $\eta^{r+1}=\beta^{r+1}=0$ hold.

\begin{theorem}\label{thm SY}(Shen--Yin \cite{SY21}) Let the notation be as above. Then the following holds:
%With the above notation the following vanishings $\eta^{r+1}=\beta^{r+1}=0$ hold. 
%Moreover  the following holds:
\begin{enumerate}
\item\label{good SY} The pair $(f,\eta)$ admits a good primitive decomposition (see \cref{def de cataldo}) 
\[ H^*(X)= \bigoplus_{0\leq i \leq r} \bQ[\eta]/(\eta^{r-i+1}) \otimes W^i\ . \]
\item\label{main SY}  As  $\bQ[\eta,\beta]$-module, $H^*(X)$ can be written as
\[ H^*(X)= \bigoplus_{0\leq i,j \leq r} \bQ[\eta]/(\eta^{r-i+1}) \otimes \bQ[\beta]/(\beta^{r-j+1}) \otimes V^{i,j}\ \]
for certain subspaces $V^{i,j}\subset H^{i+j}:=H^{i+j}(X)$,
\item There is an equality $W^i=\bigoplus_{0\leq j\leq r} \bQ[\beta]/(\beta^{r-j+1}) \otimes V^{i,j}$.
\end{enumerate}
\end{theorem}
\begin{proof}
The vanishings are explained in \cite[(22)]{SY21}. 
The existence of $V_{i,j}$ is explained in \cite[Theorem 3.1 and its proof]{SY21}. 
The relation with the good primitive decomposition is in \cite[Theorem 3.1 and its proof]{SY21}. 
\end{proof}

\begin{remark}\label{cor unicity}
The spaces $V_{i,j}$ as in \cref{thm SY} \eqref{main SY} are unique.
Indeed, one has 
  \[V_{i,j}= H^{i+j} \cap \Ker(\cup \eta^{r-i+1}) \cap \Ker(\cup \beta^{r-j+1})\ .\]
\end{remark}

\begin{proposition}\label{decomp eta}
Keep the notation as above and suppose that the fibration $f$ is an Ng\^o fibration whose fibers are all irreducible.
Then the good primitive decomposition in \cref{thm SY}\eqref{good SY} is motivic and each motive in the decomposition is isomorphic to its dual (up to a Tate twist). 
 \end{proposition}
 \begin{proof}
 By \cref{relative SCLT HK} such a Lagrangian fibration verifies the relative Lefschetz standard conjecture. 
 We can then apply \cref{motivic good decomposition}.
 \end{proof}

%\begin{corollary}
%Suppose that there exists a HK deformation $X'$ of $X$ such that $X'$ admits a lagrangian fibration which is a Ng\^o fibration with irreducible fibers, and a relatively ample class  $\eta'$ specializing to  $\eta$.
%Then the good primitive decomposition in \cref{thm SY}\eqref{good SY} is motivic and each motive in the decomposition is isomorphic to its dual (up to Tate twist).
%\end{corollary}

%\begin{remark}
%Recall some examples where this apply.
%\end{remark}
 Recall that birational HK varieties have (non-canonically) isomorphic cohomology rings and motives \cite{Huy,Riess}, see Remark \ref{rem huy}.
This is implicit in the following statement.
\begin{proposition}\label{decomp two fibrations}  Let $\eta$ and $\beta$ be as above.
Suppose that:
\begin{enumerate}
\item There exists a  HK variety $Y$ birational to $X$ together with a Lagrangian fibration $g:Y \rightarrow B'$ such that $\beta$ is relatively ample for $g$ and $\eta$ is a Lagrangian class for $g$,
\item Both $f$ and $g$ are Ng\^o fibrations with irreducible fibers (or can be deformed to fibrations with this property).
\end{enumerate}
 Then  the  decomposition in \cref{thm SY}\eqref{main SY} is motivic and each motive in the decomposition is isomorphic to its dual (up to a Tate twist).
In particular the standard conjectures hold true for $X$.
 \end{proposition}
 \begin{proof}
The Shen--Yin decomposition from \cref{thm SY}\eqref{main SY} is independent of the choice of $f$ and $g$, and only depends on $\eta$ and $\beta$  (see \cref{cor unicity}).

By \cref{thm SY}\eqref{good SY} if we forget the action of $\beta$ then we get the good decomposition with respect to $(f,\eta)$ and the latter is motivic by \cref{decomp eta}.
Similarly we can forget the action of $\eta$ and deduce that it is motivic as it is the good decomposition with respect to $(g,\beta)$.

If we now put these two decompositions together we deduce that  the decomposition in \cref{thm SY}\eqref{main SY} is motivic.
Moreover, as each factor of the two decompositions is selfdual we deduce that the factors in the Shen--Yin decomposition from \cref{thm SY}\eqref{main SY}      are so.
On the other hand the Shen--Yin decomposition refines the K\"unneth decomposition. Hence we have shown that the K\"unneth decomposition is motivic and that each of its factors is selfdual. This is one of the equivalent formulations of the Lefschetz standard conjecture.
 \end{proof}

\section{Voisin's trick and the end of the proof} \label{section voisin}

 We keep the notation from the previous section. In particular $X$ is a HK admitting a Lagrangian fibration and $\beta$ and $\eta$ are divisor classes on $X$, respectively Lagrangian and relatively ample,  such that $q(\beta)=q(\eta)=0$.

We have seen in \cref{decomp eta} a criterion to show that the decomposition with respect to the action of $\eta$ on the cohomology of $X$ is motivic, but a priori this does not apply to $\beta$, since $\beta$ needs to satisfy the assumptions of Proposition \ref{decomp two fibrations}. 
The goal of this section is to recall an ingenious idea of Voisin \cite{Voi20} which allows, roughly speaking, to invert the role of $\eta$ and $\beta$ (in a way similar to \cref{decomp two fibrations}). This will allow to conclude the proof of \cref{main}.

\begin{remark}\label{trick voisin}(Voisin \cite{Voi20})
Suppose that the Picard rank of $X$ is two (so that $\eta$ and $\beta$ are unique up to multiplication by a positive scalar).
Then there are two cases:
\begin{enumerate}
\item Both $\eta$ and $\beta$ are potentially Lagrangian (in the sense of \cref{def lagrangian}),
\item Only $\beta$ is potentially Lagrangian.
\end{enumerate}
In the first case, if we assume \cref{conj SYZ}, then there is a second Lagrangian fibration with respect to which the roles of $\eta$ and $\beta$ are inverted.

In the second case, (without assuming  \cref{conj SYZ}) the work of Huybrechts and Riess  \cite{Huy,Riess} implies the existence of an automorphism of the motive $\mathfrak{h}(X)$ which commutes with the cup-product and which inverts $\eta$ and $\beta$ (up to scalars). 

 By work of Boucksom and Huybrechts, the second case happens precisely when there is a prime exceptional divisor, i.e. a reduced and irreducible divisor  $\Theta$ with  $q(\Theta)<0$.
\end{remark}

\begin{theorem}\label{main thm}
Assume the hypothesis of \cref{main}.  
Then  the Shen--Yin decomposition from \cref{thm SY}\eqref{main SY} is motivic and each motive in the decomposition is isomorphic to its dual (up to Tate twist).
In particular the Lefschetz standard conjecture holds true for $X$.  
 \end{theorem}
 \begin{proof}
 We start with Voisin's idea explained in \cref{trick voisin} and consider the two cases listed there.
 In the first case, one is precisely in the setting of  \cref{decomp two fibrations}.  Hence, let us treat the second case.
 
 Let $\gamma$ be the automorphism of the motive $\mathfrak{h}(X)$ as in \cref{trick voisin}.
 The conjugation by $\gamma$ sends a projector to a projector hence it sends
the decomposition in \cref{thm SY}\eqref{main SY} to a new decomposition.

Moreover, as $\gamma$ commutes with the cup-product and  inverts $\eta$ and $\beta$ (up to scalar) one has the following identities (up to scalar)
\begin{align}\label{eq conj 1}
\gamma^{-1}\circ \cup \eta \circ \gamma = \cup \beta,\hspace{1cm}\gamma^{-1}\circ \cup \beta \circ \gamma = \cup \eta.
\end{align}
Hence the new decomposition obtained after conjugation is
\[ H^*(X)= \bigoplus_{0\leq i,j \leq r} \bQ[\beta]/(\beta^{r-i+1}) \otimes \bQ[\eta]/(\eta^{r-j+1}) \otimes (\gamma^{-1}V^{i,j}).\]
By the unicity of the decomposition (see \cref{cor unicity}) one gets 
\begin{align}\label{eq conj 2}
\gamma^{-1}V^{i,j}= V^{j,i}.
\end{align}

Now, let us consider the decomposition in \cref{thm SY}\eqref{main SY}. We want to show that this decomposition is motivic and that each motive in the decomposition is selfdual. Arguing as in the proof of \cref{decomp two fibrations}
it is enough to consider the two induced decompositions (obtained by forgetting the action of $\beta$ respectively $\eta$) and show that they have the same properties.

The first one is achieved, as in the proof of \cref{decomp two fibrations}, by combining  \cref{thm SY}\eqref{good SY} with \cref{decomp eta}.
The second one is deduced from the first one because  it is obtained from  the first one by conjugating by $\gamma$ (as is shown by the relations \eqref{eq conj 1} and \eqref{eq conj 2}).

We have shown that the Shen--Yin decomposition is motivic. This implies the   Lefschetz standard conjecture as we explained in the end of the proof of \cref{decomp two fibrations}.
 \end{proof}
 
\begin{remark}
The main application of the above theorem is given in Section \ref{S:LSV} and proves the Lefschetz standard conjecture for LSV tenfolds. It should be possible to extend this result to some of their deformations via a specialization argument.
%through \cref{extension}.
\end{remark}

\section{An application to LSV tenfolds}\label{S:LSV}

In this section we prove \cref{thm LSV intro}, as a concrete application of Theorem \ref{main}. We actually prove an extension  of it,  namely Theorem \ref{thm LSV} below.

The section is divided in three subsections. The first subsection recalls the geometry of a general LSV tenfold and identifies the abelian group scheme acting on it;
the second subsection proves the polarizability of this  group scheme; the third subsection gives the proof of the Lefschetz standard conjecture for all LSV tenfolds (Theorem \ref{thm LSV})  and their twisted analogues. 

Let us recall what we mean by \emph{LSV tenfolds}. Let $M$ be a smooth cubic fourfold and let $Y \rightarrow B$ be the universal family  of cubic threefolds over $B:= (\mathbb{P}^5)^{\vee}$ whose fibers are the hyperplane sections of $M$. Let $U \subset B$ be the locus parametrizing smooth hyperplane sections and let $J_U \rightarrow U$ be the intermediate Jacobian fibration. By Donagi-Markman the total space of this fibration has a holomorphic symplectic form.
By \cite{LSV17, Sacca}, there exists a projective HK manifold $X$ of OG$10$-type which compactifies $J_U$, and which is equipped with a Lagrangian fibration  $X \to B$ extending the intermediate Jacobian fibration. In \cite{LSV17}, this is done for a  \emph{general}  cubic fourfold $M$ using an explicit construction that will be recalled below. In \cite{Sacca} this is extended to any cubic fourfold $M$ (smooth or mildly singular). These Lagrangian fibered HK manifolds compactifying $J_U \to U$ will be referred to as LSV tenfolds. If $M$ is general, we will refer to the corresponding HK as a general LSV tenfold. Note that the HK manifolds constructed in \cite{Sacca} are specializations of general LSV tenfolds. While for general $M$ a Lagrangian fibered HK compactification of $J_U$ is unique, for non-general
  $M$ there may be several non isomorphic compactifications. Note, however, that they are all birational.
In \cite{Voisin-twisted},  Voisin constructs a Lagrangian fibered HK compactification of the so-called \emph{twisted} fibration $J^t_U \to U$ associated to a general $M$. This is a non-trivial torsor over $J_U$, and the fibers parametrize one-cycles of degree one on the hyperplane sections, modulo rational equivalence. In \cite{Sacca} this is extended to arbitary $M$.

We show that the criterion established in Theorem \ref{main} applies to general LSV tenfolds, as well as to their twisted analogues. From this, Theorem \ref{thm LSV} follows by a specialization argument.

%In this section, we provide a concrete application of  Theorem \ref{main}: we show that our criterion applies to the so-called \emph{LSV tenfolds}, i.e. the HK varieties of OG10-type constructed by Laza--Sacc\`a--Voisin \cite{LSV17}. In particular, the standard conjectures hold true for these varieties.

\subsection{LSV tenfolds and their Ng\^o fibration} Let $M$ be a general cubic fourfold. Let us briefly recall how the LSV tenfolds associated to $M$ and their associated Ng\^o fibration are constructed. We keep the notation introduced above and let $f: X \to B$ denote the Lagrangian fibration on the HK compactification of $J_U$ constructed in \cite{LSV17}.
%The construction starts by considering a general cubic fourfold $M$ in $\mathbb{P}^5$ and the universal family $Y \rightarrow B$ of cubic threefolds over $B:= (\mathbb{P}^5)^{\vee}$, whose fibers are obtained as hyperplane sections of $M$. If $U$ denotes the smoothness locus in $B$ of the projection $Y \rightarrow B$, the LSV tenfold $X$ associated to $M$ is then  obtained as a suitable smooth compactification $X$ over $B$ of the intermediate Jacobian fibration $J_U \rightarrow U$ attached to the restriction of $Y \rightarrow B$ to $U$. This compactification $X$ turns out to be a HK variety of OG10-type, and the resulting morphism $f:X \rightarrow B$ is a lagrangian fibration. 

For each $b \in B$, one considers the family of lines contained in the fiber $Y_b$ of $Y \rightarrow B$, thus obtaining a relative Fano surface $\mathcal{F} \rightarrow B$. In \cite{LSV17}, an open subscheme $\mathcal{F}^0 \rightarrow B$ of $\mathcal{F}$, called subscheme of \emph{very good lines}, is defined, via \emph{loc. cit.}, Definition 2.9. We obtain the smooth commutative group scheme $g:P \rightarrow B$ acting on $X \rightarrow B$ by descent to $B$ of a relative Prym variety over $\mathcal{F}^0$, i.e. the relative Prym attached to a certain \'etale double cover $\tilde{\mathcal{C}} \rightarrow \mathcal{C}$ of relative curves over $\mathcal{F}^0$. Let us describe this more precisely. 

If one fixes a very good line $\ell \in \mathcal{F}^0$ over $b \in B$, projection from $\ell$ realizes the cubic threefold $\mathcal{Y}_b$ as a conic bundle over $\mathbb{P}^2$ with discriminant a quintic curve $\mathcal{C}_{\ell}$: the latter is precisely the fiber of $\mathcal{C}$ over $\ell$. The double cover $\tilde{\mathcal{C}}_{\ell} \rightarrow \mathcal{C}_{\ell}$ is defined as the curve parametrizing the lines in $\mathcal{Y}_b$ incident to $\ell$. By definition, if the line is very good then for each $\ell$, both $\tilde{\mathcal{C}}_{\ell}$ and $\mathcal{C}_{\ell}$ are reduced and irreducible. 

We have thus an étale double cover 
\begin{equation} \label{double cover}
p:\tilde{\mathcal{C}} \rightarrow \mathcal{C}
\end{equation}
of relative curves over $\mathcal{F}^0$ at our disposal.

\begin{definition} \label{jacobian}
We will denote by $\tilde{\mathcal{J}}$, $\mathcal{J}$ the identity components of the relative Picard schemes of $\tilde{\mathcal{C}}$ and $\mathcal{C}$ over $\mathcal{F}^0$. 
\end{definition}
\begin{remark}
Since $p$ is étale and the coherent sheaf $p_* \mathcal{O}_{\tilde{\mathcal{C}}}$ is locally free, we can define a \emph{relative norm map} (\cite[6.5]{EGAII})
\begin{equation} \label{norm map}
 \mbox{Nm}_p : \tilde{\mathcal{J}} \rightarrow \mathcal{J}
\end{equation} 
which verifies in particular the following property (\cite[Prop. 6.5.8]{EGAII}). For each closed point $s$ in $\mathcal{F}^0$, consider the commutative diagram 
 \begin{equation} \label{normalizations}
\begin{tikzcd}
\tilde{\mathcal{C}}_s  \arrow[r, "p_s"]& \mathcal{C}_s \\
\tilde{D}_s \arrow[u, "\tilde{n}"] \ar[r,"\epsilon"] & D_s \ar[u,"n"]
\end{tikzcd}
 \end{equation}
where $n$, $\tilde{n}$ are normalization morphisms and $\epsilon$ is again an étale double cover. Observe that the diagram is in fact cartesian, because normalization commutes with smooth base change. Then the associated norm morphisms $\mbox{Nm}_{p_s}:\tilde{\mathcal{J}}_{s} \rightarrow \mathcal{J}_{s}$ and $\mbox{Nm}_{\epsilon}:J_{\tilde{D}_s} \rightarrow J_{D_s}$ are compatible, in the sense that 
\begin{equation} \label{norm_comp}
n^* \circ \mbox{Nm}_{p_s} = \mbox{Nm}_{\epsilon} \circ \tilde{n}^*\ .
\end{equation}
\end{remark}
We can now give the following: 
\begin{definition} \label{Prym}
The relative Prym $\mathcal{P} \rightarrow \mathcal{F}^0$ associated to the \'etale double cover $p$ \eqref{double cover} is the identity component of the kernel $\ker(\mbox{Nm}_{p})$ of the norm map \eqref{norm map}.
\end{definition}
The proof of \cite[Lemma 4.9]{LSV17} shows that the above relative Prym can be equivalently defined as the identity component of the fixed locus of the involution $\tau$ on $\tilde{\mathcal{J}}$ defined as
  \[ \tau:= -\iota^\ast\colon\ \ \tilde{\mathcal{J}}\ \to\ \tilde{\mathcal{J}}\ ,\] 
  where $\iota$ is the involution of $\tilde{\mathcal{C}}$ defined by the double cover $ \tilde{\mathcal{C}}\to {\mathcal{C}}$. 

\begin{lemma} \label{descent}
The group $\mathcal{F}^0$-scheme $\mathcal{P} \rightarrow \mathcal{F}^0$ descends along the projection $\mathcal{F}^0 \rightarrow B$ to a smooth, commutative group $B$-scheme $g: P \rightarrow B$ with an action $P \times_B X \rightarrow X$ over $B$.   Moreover, $P$ is isomorphic over $B$ to the smooth locus $X^{\mbox{\tiny{sm}}}$ of $X \rightarrow B$.
\end{lemma}
\begin{proof}
The fibers of the morphism $f:X \rightarrow B$ are integral, since   each such fiber can be identified with the closure in $X$ of a (smooth, connected) Prym variety (cfr. \cite[Proposition 4.10 (1), proof of Corollary 4.17]{LSV17}). The morphism $f$ is then part of an \emph{integrable system} in the sense of \cite[Definition 2.6]{AF16}, and by applying \cite[Theorem 2]{AF16}, we obtain a smooth, commutative $B$-group scheme $P \rightarrow B$ with an action $P \times_B X \rightarrow X$ over $B$, such that the smooth locus $X^{\mbox{\tiny{sm}}}$ of $f$ in $X$ is a torsor over $P$. Since by \cite[Corollary 4.8]{Sacca} $X^{\mbox{\tiny{sm}}} \to B$ has a zero section, it follows that $P \simeq X^{\mbox{\tiny{sm}}}$.
Now remember that by construction (cfr. \cite[Section 5]{LSV17}), $f:X \rightarrow B$ is part of a cartesian square
\[
\begin{tikzcd} 
X^\prime \arrow[r, "f^\prime"] \arrow[d] & \mathcal{F}^0 \arrow[d] \\
X \arrow[r, "f"] & B
\end{tikzcd}
\]
where $X^\prime$ is a compactification of $\mathcal{P}$ over $\mathcal{F}^0$, and more precisely, $\mathcal{P}$ coincides with the smooth locus of $f^\prime$ in $X^\prime$. This tells us that we have a cartesian diagram 
\[
\begin{tikzcd} 
\mathcal{P} \arrow[r] \arrow[d] & \mathcal{F}^0 \arrow[d] \\
X^{\mbox{\tiny{sm}}} \arrow[r] & B
\end{tikzcd}
\]
 % By base-changing to $\mathcal{F}^0$ the $P$-action on $X^{\mbox{\tiny{sm}}}$, $\mathcal{P}$ becomes a $P_{\mathcal{F}^0}$-torsor. As the group scheme $\mathcal{P}$ has a zero section, this $P_{\mathcal{F}^0}$-torsor is necessarily trivial, so that we have an isomorphism\footnote{This isomorphism is even an isomorphism of \emph{group schemes}. In fact, it suffices to check this after restriction to some open dense subgroup schemes. Hence, let $U$ denote the open subset of $B$ parametrizing hyperplane sections which are smooth. Then, the pullbacks to $\mathcal{F}_U^0$ of the group schemes in question are in fact abelian schemes, and the restriction of our isomorphism is a group scheme isomorphism since, by construction, it preserves the zero sections. \MC{Other possibility: it should be true that this restriction coincides with Mumford's (group scheme) isomorphism between a relative Prym and a relative intermediate Jacobian, cfr, \cite[proof of Prop. 5.1]{LSV17}}} $P_{\mathcal{F}^0} \simeq \mathcal{P}$,  i.e. $P$ is a descent of $\mathcal{P}$ over $B$. 
 hence concluding the proof.
\end{proof}

%\MC{To do: remark about the possibility of proving that actually $P \simeq X^{\mbox{\tiny{sm}}}$, using the existence of a section of $X \rightarrow B$ for a (general?) LSV tenfold X.} \GS{I added a remark in the text above and also in the statement of the Lemma}

Since we want to apply our Theorem \ref{main} to an LSV tenfold $X$, it becomes our task to show: 

\begin{proposition} \label{LSV_Ngo} The data $(X, B, P)$ associated as above to a general LSV tenfold $X$ define an Ng\^o fibration (i.e. a $\delta$-regular weak abelian fibration, see \cref{WAF}) with integral fibers.
\end{proposition}

\begin{proof}
By construction, our data $(X, B, P)$ satisfy the first four conditions in Definition \ref{WAF}. Now, since $P$ has been constructed in the proof of Lemma \ref{descent} by applying \cite[Theorem 2]{AF16}, we have that $P \rightarrow B$ is automatically part of the data defining a \emph{degenerate abelian scheme} in the sense of  Definition 2.1 of \emph{loc. cit.}. In particular, by  Definition 2.1(vi) of \emph{loc. cit.}, for every schematic point $b \in B$, there is a point $x \in X_b$ such that the stabilizer of $x$ in $P \times_B x$ is finite; but then, by \emph{loc. cit.}, Lemma 5.16, the stabilizer of every other point in the fiber over $b$ is affine, so that condition (5) in Definition \ref{WAF} is satisfied. Moreover, Definition 2.1(iv) of \emph{loc. cit.} implies that the $\delta$-regularity condition required by Definition \ref{def ngo} is satisfied, too (alternatively, the $\delta$-regularity also follows from the existence of a symplectic form on $X$, as explained in \cite[Section 2]{Ngo2}). It remains to check that condition (6) in Definition \ref{WAF} holds. This is proven in the following subsection (Lemma \ref{polarizability}).
\end{proof}

 \begin{remark} \label{Ngo twisted}
Given any general LSV tenfold, the same group scheme $P$ as in the above proposition acts, by construction, on its twisted version, in such a way that the first four conditions in Definition \ref{WAF} are satisfied. Then, the above proof applies to show that the data $(X, B, P)$ define an Ng\^o fibration with integral fibers even when $X$ is a twisted general LSV tenfold. 
\end{remark}
 \subsection{Polarizability of the descended relative Prym}

This subsection is devoted to providing the last ingredient needed in the proof of  \cref{LSV_Ngo}, namely:
\begin{lemma} \label{polarizability}
Let $P \rightarrow B$ be the commutative group scheme associated to a general LSV tenfold $X$ as in the previous subsection. Then the Tate module $T_{\ell}(P)$ is polarizable. 
\end{lemma}

In order to give the proof, we will start by constructing a polarization on the Tate module of the relative Prym $\mathcal{P}$ of   \cref{Prym}. This will be done in the next couple of lemmas.

\begin{lemma} \label{pol_jac}
Let $T_{\ell}$ be the Tate module of the relative Jacobian $\tilde{\mathcal{J}}$ associated  with the relative curve $\tilde{\mathcal{C}} \rightarrow \mathcal{F}^0$, see   \cref{jacobian}. Then $T_{\ell}$ is polarizable. 
\end{lemma}
\begin{proof}
The statement follows from the completely general construction of \cite[page 76]{Ngo10} (cf. also \cite[proof of Theorem 3.3.1]{DeC17}), which defines a polarization $W_C$  on the Tate module of the relative Picard scheme $J_C$ of a proper, flat relative curve $C$ with geometrically reduced fibers over a base $S$. The polarization is induced by a \emph{Weil pairing} 
\[
\langle \cdot, \cdot \rangle_C : J_C \times J_C \rightarrow \mbox{Pic}^0(S)
\]
constructed via the theory of the determinant of the cohomology. (The construction is first given for strictly henselian bases $S$; by applying it to the strict henselianization of the local ring of any point of our base $B$, we can then obtain a polarization over $B$). The vanishing of $W_C$ on the linear part of $T_{\ell}(J_C)$ comes from the conjunction of the following two facts: 

(i) By \cite[Lemme 4.12.2]{Ngo10}, for any geometric point $s \in S$, denoting by $\nu:D_s \rightarrow C_s$ the normalization morphism, and for any pair $L,L^\prime$ of degree-zero line bundles on $C_s$, we have the identification of $k(s)$-vector spaces
\begin{equation} \label{pairing_ab}
\langle L, L^\prime \rangle_{C_s} = \langle \nu^*L, \nu^*L^\prime \rangle_{D_s}\ .
\end{equation}

(ii) By \cite[§9.2, Corollary 11]{BLR90}, the projection $J_{C_s} \rightarrow J^{ab}_{C_s}$ on the abelian part of $J_{C_s}$, whose kernel is the linear part $J^{aff}_{C_s}$ of $J_{C_s}$, is identified with the morphism $J_{C_s} \rightarrow J_{D_s}$ induced by the normalization. 

The non-degeneracy of the polarization on $T_{\ell}(J^{ab}_{C_s})$ comes from its identification with cup product on the first étale $\ell$-adic cohomology group of $D_s$ (\cite[Chapter V, Remark 2.4(f)]{Mil80}).  
\end{proof}

\begin{lemma} \label{restr_pol}
Let $T^\prime_{\ell}$ be the Tate module of the relative Prym $\mathcal{P}$ of  \cref{Prym} and consider it as a sub-object of the Tate module $T_{\ell}$ of the relative Jacobian $\tilde{\mathcal{J}}$. Then the restriction to $T^\prime_{\ell}$ of the polarization on $T_{\ell}$ provided by Lemma \ref{restr_pol} is again a polarization.
\end{lemma}

\begin{proof}
We use the notation for relative Jacobians introduced in Definition \ref{jacobian} and the norm maps associated with the diagram \eqref{norm_comp}. Fix a closed point $s$ in $\mathcal{F}^0$ and consider the commutative diagram of exact sequences, whose rows are given by the Chevalley dévissages,
\begin{equation} \label{chevalley_diag}
\begin{tikzcd}
0 \arrow[r]& \mathcal{J}^{aff}_s \arrow[r]& \mathcal{J}_s \arrow[r, "n^*"] & \mathcal{J}_s^{ab} \arrow[r]& 0 \\
0 \arrow[r] & \tilde{\mathcal{J}}^{aff}_s \arrow[r] \arrow[u]& \tilde{\mathcal{J}}_s \arrow[r, "\tilde{n}^*"] \arrow[u, "\mbox{Nm}_{p_s}"]& \tilde{\mathcal{J}}_s^{ab} \arrow[r] \arrow[u, "\mbox{Nm}_{\epsilon}"]& 0 \\
0 \arrow[r]& \mathcal{P}^{aff}_s \arrow[r] \arrow[u, "u"]& \mathcal{P}_s \arrow[r] \arrow[u, "v"]& \mathcal{P}^{ab}_s \arrow[r] \arrow[u, "w"]& 0 \\
\end{tikzcd}
\end{equation}
and where the commutativity of the upper-right square comes from \eqref{norm_comp}. The arrow $v$ is the natural inclusion, and $u$ is an inclusion, too: it is induced by the fact that $\mathcal{P}^{aff}_s$ has to map to $\tilde{\mathcal{J}}^{aff}_s$, since the latter is the largest affine subgroup scheme of $\tilde{\mathcal{J}}_s$. The arrow $w$ is then induced by the commutativity of the lower-left square. Moreover, $\ker(w)$ is finite, since it is simultaneously a subgroup of the abelian variety $\mathcal{P}_s^{ab}$ and, by the snake lemma, of the affine $\mbox{coker}(u)$. This means that passing to the corresponding Tate modules, the lower rectangle yields a commutative diagram of exact sequences

\[
\begin{tikzcd} 
0 \arrow[r]& T^{aff}_{\ell,s} \arrow[r]& T_{\ell,s} \arrow[r] & T^{ab}_{\ell,s} \arrow[r]& 0 \\
0 \arrow[r] & T^{\prime,aff}_{\ell,s} \arrow[r] \arrow[u]& T^{\prime}_{\ell,s} \arrow[r] \arrow[u]& T^{\prime, ab}_{\ell,s} \arrow[r] \arrow[u]& 0
\end{tikzcd}
\]
where the vertical arrows are inclusions. 

Now look at the restriction to $T^\prime_{\ell}$ of the polarization $W_{\tilde{\mathcal{C}}}$ on $T_{\ell}$ obtained in Lemma \ref{pol_jac}. The above diagram shows that $W_{\tilde{\mathcal{C}}_s}$ vanishes on $T^{\prime, aff}_{\ell,s}$, since it vanishes on $T^{aff}_{\ell,s}$. Hence, we need to show that the induced pairing $W^{ab}_{\tilde{\mathcal{C}}_s}$ remains non-degenerate on $T^{\prime, ab}_{\ell,s}$. 
 In order to do so, observe that with notation as in diagram \eqref{normalizations}, we have 
  \[T^{ab}_{\ell,s} \simeq T_{\ell}(J_{\tilde{D}_s})\ ,\] 
  and by formula \eqref{pairing_ab}, the induced pairing $W^{ab}_{\tilde{\mathcal{C}}_s}$ coincides with the pairing $W_{\tilde{D}_s}$. Now because of the classical relation 
  \[ \mbox{Nm}_{\epsilon} \circ \epsilon^* = 2 \cdot \mbox{Id}_{J_{D_s}}\ ,\] we have
\begin{equation}\label{ds}
T_{\ell}(J_{\tilde{D}_s}) \simeq \Ker(T_{\ell}(\mbox{Nm}_{\epsilon})) \oplus \mbox{Im}(T_{\ell}(\epsilon^*))\ .
\end{equation}
Moreover, the morphisms $\mbox{Nm}_{\epsilon}$ and $\epsilon^*$ are adjoint with respect to the bilinear forms $W_{\tilde{D}_s}$ and $W_{D_s}$ (\cite[page 186, equation (I)]{Mum70}, \cite[page 331, equation (2)]{BL04}), so that the direct sum \eqref{ds} is orthogonal with respect to $W_{\tilde{D}_s}$. The restriction of $W_{\tilde{D}_b}$ to $\Ker(T_{\ell}(\mbox{Nm}_{\epsilon}))$ thus remains non-degenerate. To get what we want, it will be enough to prove that $\Ker(T_{\ell}(\mbox{Nm}_{\epsilon}))$ coincides with $T^{\prime, ab}_{\ell,s}$. It is clear by the commutativity of \eqref{chevalley_diag} that $T^{\prime, ab}_{\ell,s}$ is included in $\Ker(T_{\ell}(\mbox{Nm}_{\epsilon}))$, and we will conclude by showing that the group scheme $\ker(\mbox{Nm}_{\epsilon}) / (\mathcal{P}_s^{ab} / \ker(w))$ is finite. 

For this, notice first that the upper-right square of \eqref{chevalley_diag} provides a surjection
\[
\ker(\mbox{Nm}_{\epsilon} \circ \tilde{n}^*) \twoheadrightarrow \ker(\mbox{Nm}_{\epsilon})
\]
hence, using the commutativity of that same square, 
%an isomorphism
%\[
%\ker(n^* \circ \mbox{Nm}_{p_s}) / \ker(\tilde{n}^*) \simeq \ker{\mbox{Nm}_{\epsilon}}
%\]
%On the other hand, the 
a surjection
\[
\ker(n^* \circ \mbox{Nm}_{p_s}) \twoheadrightarrow \ker(\mbox{Nm}_{\epsilon}) / (\mathcal{P}_s^{ab} / \ker(w))
\]
which vanishes on $\ker(\mbox{Nm}_{p_s})$, 
thus yielding a surjection
\[
\ker(n^* \circ \mbox{Nm}_{p_s}) / \ker(\mbox{Nm}_{p_s}) \twoheadrightarrow \ker(\mbox{Nm}_{\epsilon}) / (\mathcal{P}_s^{ab} / \ker(w))\ .
\]
On the other hand, by definition, the map $n^*$ vanishes on the image of the  restriction of $\mbox{Nm}_{p_s}$ to $\ker(n^* \circ \mbox{Nm}_{p_s})$. By exactness of the upper row of \eqref{chevalley_diag}, $\mbox{Nm}_{p_s}$ thus induces a monomorphism
\[
\ker(n^* \circ \mbox{Nm}_{p_s}) / \ker(\mbox{Nm}_{p_s}) \hookrightarrow \mathcal{J}_s^{aff}\ ,
\]
which shows that the group scheme $\ker(n^* \circ \mbox{Nm}_{p_s}) / \ker(\mbox{Nm}_{p_s})$ is affine. Since it surjects onto $\ker(\mbox{Nm}_{\epsilon}) / (\mathcal{P}_s^{ab} / \ker(w))$, the latter cannot have non-zero dimension, because otherwise we would get a non-trivial homomorphism from an affine group scheme to an abelian variety. 
\end{proof}

\begin{remark}
\begin{enumerate}
\item The use of the relations between the morphisms $\mbox{Nm}_{\epsilon}$ and $\epsilon^*$ in the above proof is inspired by the strategy of the polarizability proof of \cite[Lemma 4.7.1 and Thm. 4.7.2]{DeC17}, in the context of the Hitchin fibration.
\item The final diagram chase in the previous proof, showing that the group scheme $\ker(\mbox{Nm}_{\epsilon}) / (\mathcal{P}_s^{ab} / \ker(w))$ is finite, could be replaced by a more direct argument, using the involution $\iota$ on $\tilde{\mathcal{C}}$ and using that $\ker \mbox{Nm}_p$ coincides with the identity component $Fix(-\iota^*)^\circ$. In fact, the short exact sequence $0 \to \mathcal{J}^{aff}_s \to \mathcal{J}_s \to  \mathcal{J}_s^{ab} \to 0$ is compatible with the involution $\tau=-i^*$ and the corresponding involution on $\widetilde{D}$. Using the fact that abelian varieties are divisible groups, one can easily check that the short exact sequence induces a short exact sequence of identity component of fixed loci. In particular, $\mathcal{P}_s^{ab}=\ker \mbox{Nm}_{\epsilon}$ and $\ker w=0$.
\end{enumerate}
\end{remark}

We will deduce Lemma \ref{polarizability} from  \cref{restr_pol} by  showing that we can descend to $B$ the polarization over $\mathcal{F}^0$ constructed above. For this, we will need a classical fact whose proof we recall for the convenience of the reader. 

\begin{lemma} \label{ff_pullback}
Let $\pi:Z \rightarrow B$ be a smooth surjective morphism with connected fibers. Then the functor $\pi^*$ on constructible $\ell$-adic sheaves is fully faithful. 
\end{lemma}
\begin{proof}
We have canonical isomorphisms
\begin{equation} \label{adj}
\Hom_{Z}(\pi^* F,\pi^* G) \simeq \Hom_{Z}(R\pi^! F,R\pi^! G) \simeq \Hom_{B}(R\pi_! R\pi^!F, G)\ ,
\end{equation}
where the first isomorphism follows from the smoothness of $\pi$, and the second one is the adjunction isomorphism. Now, for any point $i_b : b \hookrightarrow B$, by looking at the cartesian square 
\[
\begin{tikzcd} 
Z \arrow[d, "\pi"] & \arrow[l, "i^\prime_b"] Z_b \arrow[d, "\pi_b"] \\
B & \arrow[l, "i_b"] b
\end{tikzcd}
\]
and by applying the base change isomorphism, we have
\begin{equation} \label{bchange}
i_b^* R\pi_! R\pi^!F \simeq R\pi_{b,!} \pi_b^! F_b \simeq R\pi_{b,!} \pi_b^* F_b(\delta) [2 \delta]\ ,
\end{equation}
where $F_b$ denotes the stalk of $F$ at $b$, the second isomorphism comes from the smoothness of $\pi_b$, and $\delta$ is the relative dimension of $\pi$. Observe that
\begin{equation} \label{compsupp}
R\pi_{b,!} \pi_b^* F_b \simeq R \Gamma_c (Z_b, F_b)
\end{equation}
where $\Gamma_c$ is the compactly supported global sections functor, and by our hypotheses, $Z_b$ is smooth and connected of dimension $\delta$, for any $b$. As a consequence, the complex $i_b^* R\pi_! R\pi^!F$ is concentrated in degrees $[-2 \delta, \dots, 0]$ for any $b$, hence so is $R\pi_! R\pi^!F$. Thus, as $G$ is concentrated in degree zero, 
\begin{equation} \label{H0adj}
\Hom_{B}(R\pi_! R\pi^!F, G) \simeq \Hom_{B}(H^0 (R\pi_! R\pi^!F), G)\ .
\end{equation}
Then, our conclusion will follow as soon as we will prove that the adjunction morphism 
\begin{equation} \label{trace}
H^0 (R\pi_! R\pi^!F) \rightarrow F
\end{equation}
is an isomorphism (for then, \eqref{adj} and \eqref{H0adj} will yield an isomorphism 
\[
\Hom_{Z}(\pi^* F,\pi^* G) \simeq  \Hom_B(F,G)
\]
which provides an inverse of the morphism induced by $\pi^*$). For any $b$, the morphism \eqref{trace} induces a morphism between stalks at $b$, that by \eqref{bchange} and \eqref{compsupp} takes the form
\[
H^0 \left( R \Gamma_c (Z_b, F_b(\delta))[2 \delta] \right)= H_c^{2 \delta} (Z_b, F_b)(\delta) \rightarrow F_b
\]
By \cite[Exp. XVIII, Thm. 3.2.5]{SGA4}, this morphism coincides with the trace morphism on top-dimensional compactly supported cohomology. Since by our hypotheses on $\pi$, the fiber $Z_b$ is smooth, connected and non-empty for all $b$, the trace morphism is an isomorphism for all $b$.
\end{proof}

\begin{proof}[Proof of \cref{polarizability}]

The group scheme $P$ is a descent to $B$ of the relative Prym $\mathcal{P}$ over $\mathcal{F}^0$. We will show that the polarization defined in Lemma \ref{restr_pol} descends to a polarization of the Tate module of $P \rightarrow B$. 

Denote by $W^\prime$ the polarization on the Tate module $T_{\ell}(\mathcal{P})$ of $\mathcal{P}$ constructed in Lemma \ref{restr_pol}. 
Consider the cartesian square 
\[
\begin{tikzcd} 
\mathcal{P} \arrow[r, "g^\prime"] \arrow[d] & \mathcal{F}^0 \arrow[d, "\pi"] \\
P \arrow[r, "g"] & B
\end{tikzcd}
\]
If $S$ is a base scheme and $F,G$ are two constructible $\ell$-adic complexes on $S$, denote $\Hom_S(F,G):=\Hom_{D^b(S)}(F,G)$.
Observe then that by the definition of $T_{\ell}(\mathcal{P})$, the pairing $W^\prime$ is an element of
\[
\Hom_{\mathcal{F}^0}\bigl(R^{2d-1} g^\prime_!\overline{\mathbb{Q}}_{\ell,\mathcal{P}}(d) \otimes R^{2d-1} g^\prime_!\overline{\mathbb{Q}}_{\ell,\mathcal{P}}(d), \overline{\mathbb{Q}}_{\ell,\mathcal{F}^0}(1)\bigr)\ .
\]
By applying the base change isomorphism and the monoidality of pullback, we see that the above space of morphisms actually coincides with
\[
\Hom_{\mathcal{F}^0}\bigl(\pi^* \left( R^{2d-1} g_!\overline{\mathbb{Q}}_{\ell,P}(d) \otimes R^{2d-1} g_!\overline{\mathbb{Q}}_{\ell,P}(d) \right), \pi^* \overline{\mathbb{Q}}_{\ell,B}(1)\bigr)\ .
\]
But since the morphism $\pi$ is surjective and smooth (\cite[Section 5, page 110]{LSV17}) and has connected fibers (as follows from \cite[Lemma 2.5 (ii)]{LSV17}), Lemma \ref{ff_pullback} implies that the functor $\pi^*$ between constructible sheaves is fully faithful. 
In other words, for any two constructible sheaves $F$ and $G$ on $B$, the functor $\pi^*$ induces an isomorphism
\begin{equation} 
\Hom_B(F,G) \stackrel{\pi^*}{\simeq} \Hom_{\mathcal{F}^0}(\pi^* F,\pi^* G) \ .
\end{equation}
This tells us that $W^\prime$ arises from a pairing 
\[
W \in \Hom_{B}\bigl(R^{2d-1} g_!\overline{\mathbb{Q}}_{\ell,P}(d) \otimes R^{2d-1} g_!\overline{\mathbb{Q}}_{\ell,P}(d), \overline{\mathbb{Q}}_{\ell,B}(1)\bigr)\ .
\]
It remains to check that $W$ is a polarization. 
Take $b \in B$ a closed point and $s \in \mathcal{F}^0$ a closed point in the fiber $\mathcal{F}^0_b$ over $b$. 
We already know that the null space of the pairing
\[
W_s^\prime \in \Hom(T_{\ell}(\mathcal{P})_{s} \otimes T_{\ell}(\mathcal{P})_{s}, \overline{\mathbb{Q}}_{\ell}(1))
\]
is precisely the Tate module $T_{\ell}(P^{\prime,aff})_{s}$ of the linear part of $P_s^\prime$, and we want to deduce the same for  
\[
W_b \in \Hom\bigl(T_{\ell}(P)_{b} \otimes T_{\ell}(P)_{b}, \overline{\mathbb{Q}}_{\ell}(1)\bigr)\ .
\]
But observe that the natural isomorphism of fibers $P_b  \simeq \mathcal{P}_s$ induces a natural isomorphism
\[
\Hom\bigl(T_{\ell}(P)_{b} \otimes T_{\ell}(P)_{b}, \overline{\mathbb{Q}}_{\ell}(1)\bigr) \stackrel{I}{\simeq} \Hom\bigl(T_{\ell}(\mathcal{P})_{s} \otimes T_{\ell}(\mathcal{P})_{s}, \overline{\mathbb{Q}}_{\ell}(1)\bigr)\ ,
\]
and we claim that $W_b$ maps to $W_s^\prime$ under this isomorphism, thus concluding the proof.

 In fact, for any two constructible sheaves $F,G$ on $B$, the cartesian square
\[
\begin{tikzcd} 
\mathcal{F}^0 \arrow[d, "\pi"] & \arrow[l, "i^\prime_b"] \mathcal{F}^0_b \arrow[d, "\pi_b"] \\
B & \arrow[l, "i_b"] b
\end{tikzcd}
\]
yields a commutative diagram, whose subdiagrams all commute,
\[
\begin{tikzcd}
\Hom_B(F,G) \arrow[d, "\pi^*"] \arrow[r, "i^*_b"] & \Hom(F_b,G_b)  \arrow[d,"\pi^*_b"]  \arrow[dr, "I"] & \\
\Hom_{\mathcal{F}^0}(\pi^* F,\pi^* G) \arrow[r, "i_b^{\prime,*}"] & 
\Hom_{\mathcal{F}_b^0}(\pi_b^* F_b,\pi_b^* G_b) \arrow[r, "\iota"] & \Hom((\pi^*F)_s, (\pi^*G)_s)
\end{tikzcd}
\]
Here the vertical arrows are isomorphisms by fully faithfulness, and $\iota$ is pullback to $s \in \mathcal{F}_b^0$, also an isomorphism since the sheaves $\pi_b^*F_b$ and $\pi_b^*G_b$ are constant of stalk $F_b \simeq (\pi^*F)_s$ resp. $G_b \simeq (\pi^*G)_s$ on the fiber $\mathcal{F}_b^0$. Now, setting $F=T_{\ell}(P) \otimes T_{\ell}(P)$ and $G=\overline{\mathbb{Q}}_{\ell,B}$, this implies that $I$ maps $W_b$ to the image under $\iota \circ i^{\prime,*}_b$ of $W^\prime = \pi^* W$. The latter being precisely $W^\prime_s$, this proves our claim. 
\end{proof}

\begin{remark} 
The paper \cite{AF}, which appeared subsequently to this work, proved a polarizability result in  greater generality, implying in particular  \cref{polarizability}. The proof we present here is more geometric   and we believe of independent interest.
\end{remark}

\subsection{Application of Theorem \ref{main} to LSV tenfolds}

 In this subsection we show \cref{thm LSV}, which is an extension of  \cref{thm LSV intro} from the Introduction.
\begin{theorem}\label{thm LSV}
Let $X$ be a LSV tenfold, respectively $X^t$  its twisted version. Then the standard conjectures hold true for $X$, respectively for $X^t$. 
\end{theorem}

\begin{proof} Any given LSV tenfold $X$ is part of a family, where the general member is a general LSV tenfold $X^\prime$ with Picard number 2 (this is just because the cubic fourfold $M$ defining $X$ can be deformed to a cubic fourfold $M^\prime$ which is general in the sense of \cite{LSV17} and with $h^{2,2}(M^\prime)=1$). Via a specialization argument, one is then reduced to proving the theorem for $X^\prime$. To this end, let us check that the assumptions of Theorem \ref{main} are satisfied by $X^\prime$. (The analogous specialization argument holds for the twisted version.)

Point (1) of Theorem \ref{main} is clear, and point (2) follows from Proposition \ref{SYZ holds}. As for point (3), it is known that $X^\prime$ has only one primitive Lagrangian class (\cite[Corollary 3.10]{Sacca}, which follows from the existence of an effective divisor $\Theta$ with $q(\Theta)<0$), and so we are in the second case of Remark \ref{trick voisin}. The same holds for the twisted version, using the fact that by   \cite[Theorem 1.3]{LPZ}, $X^t$ is birational to a divisorial symplectic resolution, hence the second ray of the birational K\"ahler cone determines a divisorial contraction and not a Lagrangian fibration. But then, in both cases, point (3) reduces to checking that the Lagrangian fibration $X^\prime\to B$ described in \cite{LSV17} (or its twisted analogue) is an Ng\^o fibration; this is the content of Proposition \ref{LSV_Ngo} and \cref{Ngo twisted}.
\end{proof}

\begin{remark}
As explained in the proof, an LSV tenfold $X$ with Picard number 2 falls into the second case of Remark \ref{trick voisin}, so that there is only one Lagrangian fibration to study, namely the one naturally provided by the construction of $X$. 
\end{remark}

%\begin{remark} The construction of LSV tenfolds in \cite{LSV17} is done under the hypothesis that the cubic fourfold is {\em general\/} (in a certain precise sense). Recently, it has been shown in \cite{Sacca} that this construction can be extended, in the sense that to {\em any\/} smooth cubic fourfold one can associate a HK variety of OG10-type (which is a smooth compactification of the intermediate Jacobian fibration, and which coincides with the construction of \cite{LSV17} for general cubic fourfolds). 
%
%All of these HK tenfolds constructed in \cite{Sacca} satisfy the standard conjectures. Indeed, the construction of \cite{Sacca} can be done in a family where the general member is a LSV tenfold (i.e. associated to a cubic which is general in the sense of \cite{LSV17}). Theorem \ref{thm LSV} combined with a specialization argument then proves that the standard conjectures hold for all members of the family.
%\end{remark}

 \begin{remark} 
%There is another family of HK varieties of OG10-type constructed in \cite{Voisin-twisted}, the so-called {\em twisted intermediate Jacobian fibrations} (the construction is a variant of that of the LSV tenfolds; in the twisted case, a general fiber of the lagrangian fibration is identified with one-cycles of degree one on a hyperplane section, modulo rational equivalence). It is known that a general twisted intermediate Jacobian fibration is {\em not\/} birational to a LSV tenfold \cite[Corollary 3.10]{Sacca}.
%The above argument would seem to apply to this family as well. 
The following more direct approach is possible to prove the standard conjectures for twisted intermediate Jacobian fibrations.
%However, a rather more direct approach to the standard conjectures for twisted intermediate Jacobian fibrations is possible.

Let $Y=X^t$ be the twisted intermediate Jacobian fibration associated to a very general cubic fourfold $Z\subset\bP^5$. As mentioned earlier, it is shown in \cite[Theorem 1.3 and Remark 6.9]{LPZ} that
$Y$ is birational to the HK variety $\tilde{M}$ obtained as a symplectic resolution of a certain moduli space of semi-stable objects in the Kuznetsov category of $Z$.
Applying \cite[Theorem 1.8]{FFZ}, the Chow motive of $\tilde{M}$ is in the subcategory generated by the motive of $Z$. As birational HK varieties have isomorphic Chow motives, it follows that there is an inclusion of (Chow and hence also of homological) motives
 \begin{equation}\label{submotive}  \mathfrak{h}(Y)\cong \mathfrak{h}(\tilde{M})\ \hookrightarrow\ \bigoplus \mathfrak{h}(Z^r)(\ast)\ \ \hbox{in}\   \Mot_{} \ .\end{equation}
In particular, $Y$ verifies the standard conjectures. 

The same remark applies to the HK varieties of OG6 type constructed as symplectic resolutions of moduli spaces of semi-stable sheaves on an abelian surface $A$ in \cite{OGrady03}. As shown in \cite{Wu}, these HK varieties $Y$ admit a Lagrangian fibration that is an Ng\^o fibration (with possibly reducible fibers), and so perhaps a stratified version of our argument might apply. However, it is known that the Chow motive of $Y$ is in the subcategory generated by the motive of $A$ \cite{Floc}, and so in particular $Y$ verifies the standard conjectures.

On the other hand, we observe that this line of reasoning does {\em not\/} apply to the LSV tenfolds. Indeed, it is known that the very general LSV tenfold is {\em not\/} birational to a moduli space of stable objects in the Kuznetsov category of $Z$ \cite[Corollary 4.2]{Sacca} (cf. also \cite[Remark 4.3]{Sacca}).
Nevertheless, one may wonder whether the inclusion \eqref{submotive} might still be true for the general LSV tenfold.
\end{remark}

\bibliographystyle{alpha}%a mix between alpha and abbrv - shows letters for references and abbreviates the first names to initials
\bibliography{MyLibrary}

%\printbibliography
\Addresses

\end{document}
 \begin{proposition}\label{robert trick}
  Let $X$ be a HK variety which admits a Lagrangian fibration $f:X \rightarrow B$.
Assume that:
   \begin{enumerate}
   \item $X$ is of OG10-type.
\item The Picard rank of $X$ is $2$.

\end{enumerate}
Then there exists a HK variety $X^\prime$, birational to $X$ over $B$ and admitting a unique Lagrangian fibration and a unique lagrangian class (\cref{def lagrangian}).
   \end{proposition}

\begin{proof}
 This is contained (more or less explicitly) in \cite[Section 3]{quartet}. Since $X$ is of OG10-type, \cite[Example 3.5]{quartet} implies that the numerical assumptions of \cite[Theorem 3.1]{quartet} are satisfied by $X$. Then, as explained in \cite[Proof of Theorem 3.1]{quartet}, the work of Matsushita \cite{Matsushita} gives a HK variety $X^\prime$, birational to $X$ over $B$, and equipped with a line bundle $M$ having certain numerical properties. Next, applying \cite[Section 3.2]{quartet} to $X^\prime$ and $M$, we find that $X^\prime$ has a prime exceptional divisor, and so  $X^\prime$ has a unique Lagrangian fibration (see \cref{trick voisin}).
\end{proof}
\GA{modify biration X' to X. Erase Lemma 10.2}
 \begin{lemma}\label{bir invariant} Let $f\colon X\to B$, $f^\prime\colon X^\prime\to B$ be two HK varieties of Picard number 2 with a Lagrangian fibration. Assume $X$ and $X^\prime$ are birational over $B$. Then
 $X$ verifies the relative Lefschetz standard conjecture if and only if $X^\prime$ does so.
 \end{lemma}

 \begin{proof} Let $\beta, \eta\in H^2(X)$ be the two isotropic classes as above, and let $\beta^\prime, \eta^\prime\in H^2(X^\prime)$ be the similarly defined isotropic classes on $X^\prime$. Recall that the natural isomorphism
 $H^2(X)\cong H^2(X^\prime)$ respects the underlined quadratic forms.  Hence it sends the pair $\{\bQ[\beta],\bQ[\eta]\}$ to the pair $\{\bQ[\beta^\prime],\bQ[\eta^\prime]\}$, as those are the isotropic lines.  
 Since $X$ and $X^\prime$ are birational over $B$, the class $\beta=f^\ast {\mathcal O}_B(1)$ is sent to a multiple of $\beta^\prime$. Hence, the class $\eta$ is sent to a multiple of $\eta^\prime$.
 
 Let us assume $X$ verifies the relative Lefschetz standard conjecture. Then, by Proposition \ref{motivic good decomposition}, the good decomposition for the pair $(f,\eta)$ given by Theorem \ref{thm SY} is motivic. As again the isomorphism $H^\ast(X)\cong H^\ast(X^\prime)$ is an algebra isomorphism, the good decomposition for the pair $(f,\eta)$ induces the   good decomposition for  $(f^\prime,\eta^\prime)$. In particular, as the former is motivic and selfdual, the latter is so.
   Since $\eta^\prime$ is relatively ample with respect to $f^\prime$, this means that $(X^\prime, f^\prime)$ verifies the relative Lefschetz standard conjecture. 
 \end{proof}

   \begin{theorem}\label{thm robert} 
   Let $X$ be a HK variety admitting a Lagrangian fibration $f:X \rightarrow B$.
Assume that:
   \begin{enumerate}
   \item $X$ is of OG10-type.
\item The Picard rank of $X$ is $2$.
\item The fibration $f$ is an Ng\^o fibration (see Definition \ref{def ngo}).
\item The fibers of $f$ are irreducible.
\end{enumerate}
Then the Lefschetz standard conjecture holds for $X$.  
  \end{theorem}
  
  \begin{proof} By Proposition \ref{robert trick} there exists a HK variety $X^\prime$, birational to $X$ over $B$ and admitting a unique Lagrangian class. Now, by virtue of the last two assumptions $f:X\to B$ verifies the relative Lefschetz standard conjecture (cf. Proposition \ref{relative SCLT HK}).
 Applying Lemma \ref{bir invariant}, it follows that $f^\prime : X^\prime\to B$ also verifies the relative Lefschetz standard conjecture. As $X^\prime$ has only one Lagrangian fibration,  we are in the second case of Remark \ref{trick voisin}, hence by \cref{main}  $X^\prime$ verifies the Lefschetz standard conjecture. The truth of the Lefschetz standard conjecture being a birational invariant among HK varieties, see \cref{rem huy}, it follows that $X$ also verifies the Lefschetz standard conjecture.
 \end{proof}

%% file: mynewcommands.tex
\ifpersonal
\newcommand*{\personal}[1]{\textcolor[rgb]{.2,.2,1}{\tiny{(Personal: #1)}}}
\newcommand*{\todo}[1]{\textcolor{red}{(Todo: #1)}}
\else
\newcommand*{\personal}[1]{\ignorespaces}
\newcommand*{\todo}[1]{\ignorespaces}
\fi

\usetikzlibrary{decorations.markings}

\makeatletter
\tikzcdset{
	open/.code     = {\tikzcdset{hook, circled};},
	closed/.code   = {\tikzcdset{hook, slashed};},
	open'/.code    = {\tikzcdset{hook', circled};},
	closed'/.code  = {\tikzcdset{hook', slashed};},
	circled/.code  = {\tikzcdset{markwith = {\draw (0,0) circle (.375ex);}};},
	slashed/.code  = {\tikzcdset{markwith = {\draw[-] (-.0ex,-.4ex) -- (.0ex,.4ex);}};},
	markwith/.code ={
		\pgfutil@ifundefined%
		{tikz@library@decorations.markings@loaded}%
		{\pgfutil@packageerror{tikz-cd}{You need to say %
				\string\usetikzlibrary{decorations.markings} to use arrows with markings}{}}{}%
		\pgfkeysalso{/tikz/postaction = {
				/tikz/decorate,
				/tikz/decoration={markings, mark = at position 0.5 with {#1}}}
		}
	},
}
\makeatother
\newcommand{\cM}{\mathcal M}

\newcommand{\bZ}{\mathbb{Z}}							  				
\newcommand{\bN}{\mathbb{N}}							  				
							  				
\newcommand{\bP}{\mathbb{P}}

\newcommand{\bQ}{\mathbb{Q}}

\newcommand{\bC}{\mathbb{C}}

\providecommand{\bP}{\mathbb{P}}

\newcommand{\one}{\mathbbm{1}}
\newcommand{\Mot}{{\cM ot}}

\newcommand{\Vect}{\mathrm{Vect}}

\newcommand{\CHM}{\mathrm{CHM}}

\makeatletter
\newcommand{\oset}[3][0.2ex]{%
	\mathrel{\mathop{#3}\limits^{
			\vbox to#1{\kern-2\ex@
				\hbox{$\scriptstyle#2$}\vss}}}}
\makeatother

\usetikzlibrary{decorations.markings} %arrows for open immersions and closed immersions
\tikzset{
  closed/.style = {decoration = {markings, mark = at position 0.5 with { \node[transform shape, xscale = .8, yscale=.4] {/}; } }, postaction = {decorate} },
  open/.style = {decoration = {markings, mark = at position 0.5 with { \node[transform shape, scale = .7] {$\circ$}; } }, postaction = {decorate} }
}
\DeclareMathOperator{\codim}{codim}

\DeclareMathOperator{\Hom}{Hom}
\DeclareMathOperator{\End}{End}

\DeclareMathOperator{\Ker}{Ker}

\DeclareMathOperator{\id}{\mathsf{id}}

\makeatletter
\let\save@mathaccent\mathaccent
\newcommand*\if@single[3]{%
	\setbox0\hbox{${\mathaccent"0362{#1}}^H$}%
	\setbox2\hbox{${\mathaccent"0362{\kern0pt#1}}^H$}%
	\ifdim\ht0=\ht2 #3\else #2\fi
}
\newcommand*\rel@kern[1]{\kern#1\dimexpr\macc@kerna}
\newcommand*\widebar[1]{\@ifnextchar^{{\wide@bar{#1}{0}}}{\wide@bar{#1}{1}}}
\newcommand*\wide@bar[2]{\if@single{#1}{\wide@bar@{#1}{#2}{1}}{\wide@bar@{#1}{#2}{2}}}
\newcommand*\wide@bar@[3]{%
	\begingroup
	\def\mathaccent##1##2{%
		\let\mathaccent\save@mathaccent
		\if#32 \let\macc@nucleus\first@char \fi
		\setbox\z@\hbox{$\macc@style{\macc@nucleus}_{}$}%
		\setbox\tw@\hbox{$\macc@style{\macc@nucleus}{}_{}$}%
		\dimen@\wd\tw@
		\advance\dimen@-\wd\z@
		\divide\dimen@ 3
		\@tempdima\wd\tw@
		\advance\@tempdima-\scriptspace
		\divide\@tempdima 10
		\advance\dimen@-\@tempdima
		\ifdim\dimen@>\z@ \dimen@0pt\fi
		\rel@kern{0.6}\kern-\dimen@
		\if#31
		\overline{\rel@kern{-0.6}\kern\dimen@\macc@nucleus\rel@kern{0.4}\kern\dimen@}%
		\advance\dimen@0.4\dimexpr\macc@kerna
		\let\final@kern#2%
		\ifdim\dimen@<\z@ \let\final@kern1\fi
		\if\final@kern1 \kern-\dimen@\fi
		\else
		\overline{\rel@kern{-0.6}\kern\dimen@#1}%
		\fi
	}%
	\macc@depth\@ne
	\let\math@bgroup\@empty \let\math@egroup\macc@set@skewchar
	zen@everymath{\mathgroup\macc@group\relax}%
	\macc@set@skewchar\relax
	\let\mathaccentV\macc@nested@a
	\if#31
	\macc@nested@a\relax111{#1}%
	\else
	\def\gobble@till@marker##1\endmarker{}%
	\futurelet\first@char\gobble@till@marker#1\endmarker
	\ifcat\noexpand\first@char A\else
	\def\first@char{}%
	\fi
	\macc@nested@a\relax111{\first@char}%
	\fi
	\endgroup
}
\makeatother